\documentclass[11pt,b5paper,notitlepage]{article}
\usepackage[b5paper, margin={0.5in,0.65in}]{geometry}
\usepackage{amsmath,amssymb,amsthm,amsfonts}
\usepackage{mathtools,array}
\usepackage{mathrsfs,indentfirst,mathabx}
\usepackage{palatino}

\let\OLDthebibliography\thebibliography
\renewcommand\thebibliography[1]{%
  \OLDthebibliography{#1}%
  \setlength{\parskip}{0pt}%
  \setlength{\itemsep}{2pt}%
}

\allowdisplaybreaks
\usepackage[colorlinks,linkcolor=blue,anchorcolor=blue]{hyperref}
\hypersetup{
citecolor=[rgb]{0,0.5,0},
pdftitle={Fourier coefficients of the degenerate Eisenstein series on symplectic groups},
pdfauthor={Alan Xuelun Hou and Yusheng Lei}
}
\theoremstyle{plain}
\newtheorem{theorem}{Theorem}[section]
\newtheorem{lemma}[theorem]{Lemma}
\newtheorem{proposition}[theorem]{Proposition}
\theoremstyle{definition}
\newtheorem{remark}[theorem]{Remark}

\numberwithin{equation}{section}
\newcommand{\calF}{\mathcal{F}}
\newcommand{\calO}{\mathcal{O}}
\newcommand{\calR}{\mathcal{R}}
\newcommand{\calS}{\mathcal{S}}
\newcommand{\calV}{\mathcal{V}}
\newcommand{\bbA}{\mathbb{A}}
\newcommand{\bbC}{\mathbb{C}}
\newcommand{\bbR}{\mathbb{R}}
\newcommand{\nHeis}{n_{\mathrm{H}}}
\newcommand{\nPlusKone}{n_{+,\,k+1}^{\,k-2}}
\DeclareMathOperator{\id}{{\bf 1}}
\DeclareMathOperator{\GL}{GL}
\DeclareMathOperator{\Sp}{Sp}
\DeclareMathOperator{\Hom}{Hom}
\DeclareMathOperator{\diag}{diag}
\DeclareMathOperator{\Mat}{Mat}
\DeclareMathOperator{\Ind}{Ind}

\title{Fourier coefficients of the degenerate Eisenstein series on symplectic groups}
\author{\scshape Alan Xuelun Hou \and \scshape Yusheng Lei}
\date{}

\begin{document}\sloppy
\pagenumbering{arabic}

\maketitle
\begingroup
\renewcommand{\thefootnote}{}
\endgroup

\begin{abstract}
We study degenerate Eisenstein series on symplectic groups and construct a new automorphic descent. We show that a certain Fourier coefficient, after restriction to a smaller symplectic group, is itself an Eisenstein series with the same inducing parameter $s$. We describe the resulting holomorphic section explicitly in terms of the original section and local $L$-factors at unramified places. At every regular point of the local descent, we prove surjectivity at non-Archimedean places and dense image at Archimedean places. We finally indicate an application of the iterated descent to maximal Fourier coefficients of Eisenstein series.
\end{abstract}

\section{Introduction}
\subsection{Background}
One of the most important ways to study an Eisenstein series is by analyzing its Fourier coefficients parameterized by certain unipotent orbits. In this paper, we give a detailed analysis of such coefficients in a new situation, that of a degenerate Eisenstein series on the symplectic group.

Let $F$ be a number field and $\bbA$ its ring of adeles. Fix integers $n$ and $k$ with $3\leq k<n$. Let $P_k:=P_{k,n}=M_k\ltimes U_k$ be the standard maximal parabolic subgroup of $\Sp_{2n}$ with Levi subgroup $M_k\cong\GL_k\times\Sp_{2(n-k)}$. Let $\chi$ be a Hecke character. Extend $\chi|\det|^s\otimes\id_{\Sp_{2(n-k)}}$ trivially across $U_k(\bbA)$, and let $I_n(s)$ be its normalized parabolic induction to $\Sp_{2n}(\bbA)$. The normalization is given in (\ref{eq301}).

For a holomorphic section $f_s$ of $I_n(s)$, define the degenerate Eisenstein series, initially for $\Re(s)\gg0$, by
\begin{equation}\label{eis1}
  E_{f_s}(g) = \sum_{\gamma\in P_k(F)\setminus \Sp_{2n}(F)}f_s(\gamma g).
\end{equation}
The series $E_{f_s}$ extends meromorphically to the whole $s$-plane. It has a nonzero Fourier coefficient along the unipotent radical $U_2$ of the standard maximal parabolic subgroup $P_2$. The corresponding coefficient along $P_1$ vanishes; see \cite[Chapter~5]{GIN06}. Let $\psi_\calO$ be a generic character of $U_2(\bbA)$ as in Section~\ref{sec204}. Define
\begin{equation}\label{fc}
  F_{\psi_\calO}(E_{f_s})(g) := \int\limits_{U_{2}(F)\setminus U_{2}(\bbA)}E_{f_s}(ug)\psi_{\calO}(u)\,du.
\end{equation}

\subsection{Main theorems}
For $n\geq3$, the stabilizer of $\psi_\calO$ under conjugation by $\GL_2(\bbA)\times\Sp_{2(n-2)}(\bbA)$ is isomorphic to $\Sp_2(\bbA)\times\Sp_{2(n-3)}(\bbA)$. The coefficient $F_{\psi_\calO}(E_{f_s})$ is an automorphic function on this stabilizer. Its restriction to $\Sp_{2(n-3)}(\bbA)$ is described by the following theorem.
\begin{theorem}
For $\Re(s)$ sufficiently large, $F_{\psi_\calO}(E_{f_s})(g)$, viewed as a function on $\Sp_{2(n-3)}(\bbA)$, is an Eisenstein series associated to the parabolic induction
\[
  I_{n-3}(s) = \Ind_{P_{k-2,n-3}(\bbA)}^{\Sp_{2(n-3)}(\bbA)} \chi|\det|^{s}\otimes \id_{\Sp_{2(n-k-1)}}.
\]
At each place $v$, the local descent integrals form a meromorphic family, as in Proposition~\ref{prop4}. If $\Re(s_0)\geq 0$ and this operator family is holomorphic at $s=s_0$, then the resulting local map is surjective when $v$ is non-Archimedean and has dense image in the smooth Fréchet topology when $v$ is Archimedean.
\end{theorem}

This descent takes a degenerate Eisenstein series on $\Sp_{2n}(\bbA)$ to one on $\Sp_{2n-6}(\bbA)$ without changing $s$. At unramified places, it acts on spherical sections by an explicit scalar.
\begin{theorem}
Let $f_s=\bigotimes_v f_{s,v}$ be a factorizable holomorphic section of $I_n(s)$. At a place $v$ where $\chi_v$, $\psi_v$ and $f_{s,v}$ are unramified, assume that $f_{s,v}=f^\circ_{n,v,s}$ is the normalized spherical section of $I_{n,v}(s)$. Then
\[
  \mathcal F_{v,s}\bigl(f^\circ_{n,v,s}\bigr) = c_v(s)f^\circ_{n-3,v,s},
\]
where
\[
  c_v(s)= L_v^{-1}(\chi_v^2,2s+k-1) L_v^{-1}\left(\chi_v,s+n-\frac{k+1}{2}\right) L_v^{-1}\left(\chi_v,s+n-\frac{k-1}{2}\right).
\]
\end{theorem}

We also discuss maximal Fourier coefficients of residues. Let $\chi$ be a nontrivial quadratic Hecke character. Put $r=n-k\geq1$ and $a=3k-2n\geq2$, and let $s_0$ lie in the positive pole set $\Lambda_a$ defined in \eqref{poles}. Iterating the descent reaches a Siegel Eisenstein series. Remark~\ref{partition} explains how this iteration, together with known results on Siegel residues and composite partitions, suggests a maximal partition for suitable Laurent coefficients.

This construction belongs to the theory of automorphic descent developed in \cite{GRS972,GRS11,GS222}. Related Fourier coefficients were studied in \cite{ik94,GRS03,GIN06,Cai18,GS20}, and local models were studied in \cite{MW87}. Here the descent starts from a degenerate Eisenstein series on $\Sp_{2n}$ and produces another degenerate Eisenstein series.

Section 2 fixes the notation. Section 3 unfolds the Fourier coefficient and states the descent identity in Theorem~\ref{thm401}. Section 4 computes the local integrals using root exchange and the Casselman--Shalika and Gindikin--Karpelevich formulas. Section 5 completes the proof and discusses iteration of the descent.

\subsection{Acknowledgments}
The authors thank David Ginzburg for suggesting this problem and for helpful advice. They thank Solomon Friedberg and Omer Offen for useful discussions and suggestions. The first author also thanks Yi Shan, Jun Su and Bingyu Zhang. The second author worked on this project while he was a postdoctoral fellow at Tel Aviv University. He thanks the department for its hospitality.

\section{Preliminaries}
\subsection{Groups}
Let $F$ be a number field and let $n$ be a positive integer. Let $J_n$ be the $n \times n$ matrix with $1$ on the anti-diagonal and $0$ elsewhere. For a commutative $F$-algebra $R$, define
\[
  \Sp_{2n}(R) = \left\{g \in \GL_{2n}(R)\;|\; ^tg \begin{pmatrix}
    &J_n \\
    -J_n&
  \end{pmatrix} g = \begin{pmatrix}
    &J_n \\
    -J_n&
  \end{pmatrix} \right\}.
\]

Let $B=TU$ be the Borel subgroup of upper triangular matrices in $\Sp_{2n}$, with diagonal torus $T$ and unipotent radical $U$. For $t=\diag(t_1,\ldots,t_n,t_n^{-1},\ldots,t_1^{-1})\in T(R)$, the simple positive roots are $\alpha_i(t)=t_it_{i+1}^{-1}$ for $1\leq i<n$ and $\alpha_n(t)=t_n^2$. The root $\alpha_n$ is the unique long simple root. Let $\triangle_n$ denote the set of simple positive roots. We also denote $\alpha_i$ by $e_i - e_{i+1}$ for $i = 1,2,\cdots, n-1$ and $\alpha_n$ by $2e_n$, so that we can describe the set of all positive roots as
\[
  \Sigma^+ = \{e_i\pm e_j|1\leqslant i<j\leqslant n\}\cup \{2e_i | 1\leqslant i\leqslant n\}.
\]

Put $\Sigma=\Sigma^+\cup(-\Sigma^+)$. Let $e_{i,j}$ be the $2n\times2n$ matrix with $1$ in entry $(i,j)$ and zero elsewhere. For a positive root $\alpha$, define the root subgroup $r\mapsto X_\alpha(r)$ by
\[
  X_\alpha(r) = \begin{cases}
    \exp(r(e_{i,j} - e_{2n-j+1, 2n-i+1})) &\alpha = e_i -e_j \\
    \exp(r(e_{i,2n-j+1} + e_{j,2n-i+1})) &\alpha = e_i + e_j \\
    \exp(re_{i,2n-i+1}) &\alpha = 2e_i.
  \end{cases}
\]
We use a similar notation for the negative roots.

At a finite place $v$, put $K_v=\Sp_{2n}(\calO_v)$, where $\calO_v$ is the ring of integers of $F_v$. At a real place, take $K_v=U(n)$; at a complex place, take $K_v=USp(2n)$. Then $K=\prod_v K_v$ is a maximal compact subgroup of $\Sp_{2n}(\bbA)$. Fix a nontrivial additive character $\psi:F\backslash\bbA\to\bbC^\times$, with local factorization $\psi(x)=\prod_v\psi_v(x_v)$.

Put $G=\Sp_{2n}$, and let $X^*(\cdot)$ denote the lattice of characters defined over $F$. For a parabolic subgroup $P=MU$ of $G$, let $Z_M$ be the center of $M$ and put $\mathfrak a_P=\Hom(X^*(P),\bbR)=\Hom(X^*(M),\bbR)=\Hom(X^*(Z_M),\bbR)$. Define $H_M:M(\bbA)\to\mathfrak a_P$ by $e^{\langle\chi,H_M(m)\rangle}=|\chi(m)|$ for $m\in M(\bbA)$ and $\chi\in X^*(M)$. Write $M(\bbA)^1$ for its kernel. Extend $H_M$ to $P(\bbA)$ trivially on $U(\bbA)$; see \cite[Chapter~1]{MW95}.

Let $W_n=N_G(T)/C_G(T)$ be the Weyl group, acting on $\mathfrak a_0:=\mathfrak a_B$. If $P=M_PU_P$ and $Q=M_QU_Q$ are parabolic subgroups and $wM_Pw^{-1}=M_Q$, then $w$ restricts to an isomorphism $\mathfrak a_P\xrightarrow{\sim}\mathfrak a_Q$. We identify $W_n$ with the signed permutations of $\{1,\ldots,n\}$. Thus $w(i)=\varepsilon j$, with $\varepsilon\in\{\pm1\}$, means $w(e_i)=\varepsilon e_j$. Equivalently, $W_n$ acts on $\{\pm1,\ldots,\pm n\}$ by permutations satisfying $w(-i)=-w(i)$.

\subsection{Embeddings}
For $g=\bigl(\begin{smallmatrix}a&b\\c&d\end{smallmatrix}\bigr)\in\Sp_2(R)$ and $h\in\Sp_{2(n-3)}(R)$, define the embedding of $\Sp_2\times\Sp_{2(n-3)}$ into $\Sp_{2n}$ by
\begin{equation}\label{triangle}
  \triangle(g,h) := \begin{pmatrix}
    g&&&& \\
    &a&&b& \\
    &&h&& \\
    &c&&d& \\
    &&&&g^\ast
  \end{pmatrix} \in\Sp_{2n}(R),
\end{equation}
where $g^*$ is the element corresponding to $g$ such that $\triangle(g,h)$ is symplectic.

For $g\in\Sp_{2(n-3)}$, define the embedding $\iota$ by
\[
  g = \begin{pmatrix}
    g_1&g_2&g_3 \\
    g_4&g_5&g_6 \\
    g_7&g_8&g_9
  \end{pmatrix} \mapsto \iota(g)= \begin{pmatrix}
    g_1&&g_2&&g_3 \\
    &I_3&&0& \\
    g_4&&g_5&&g_6 \\
    &0&&I_3& \\
    g_7&&g_8&&g_9
  \end{pmatrix} ,\quad \begin{aligned}
    &g_1,g_3,g_7,g_9\in \Mat_{(k-2)\times (k-2)} \\
    &g_2,g_8 \in \Mat_{(k-2)\times 2(n-k-1)} \\
    &g_4,g_6 \in \Mat_{2(n-k-1)\times (k-2)} \\
    &g_5 \in \Mat_{2(n-k-1)\times 2(n-k-1)}.
  \end{aligned}
\]

We also use the embedding $\iota_c(g)=\diag(I_3,g,I_3)$. To indicate dependence on $k$, write $\iota=\iota_k$ for $2\leq k<n$; when $k=2$, this is $\iota_c$. For $h\in\GL_3$, define $t=t_k:\GL_3\to\Sp_{2n}$ by
\begin{equation}\label{gl3em}
  t(h) := \begin{pmatrix}
    I_{k-2}&&&& \\
    &h&&& \\
    &&I_{2(n-k-1)}&& \\
    &&&h^\ast& \\
    &&&&I_{k-2}
  \end{pmatrix} .
\end{equation}

For $(h_1,h_2)\in\GL_k\times\Sp_{2(n-k)}$, write $j(h_1,h_2)$ for its image in the Levi subgroup $M_k$.

\subsection{Degenerate Eisenstein series}
For the rest of the paper, fix $3\leq k<n$ and a Hecke character $\chi:F^\times\backslash\bbA^\times\to\bbC^\times$.

Let $I_n(s)$ be the representation of $\Sp_{2n}(\bbA)$ by right translation on smooth $K$-finite functions $f$ satisfying
\begin{equation}\label{eq301}
  f(j(h_1,h_2)ug) = \chi(\det h_1)|\det h_1|^s \delta_{P_k}^{1/2}(j(h_1,h_2))f(g)
\end{equation}
for $(h_1,h_2)\in\GL_k(\bbA)\times\Sp_{2(n-k)}(\bbA)$, $u\in U_k(\bbA)$ and $g\in\Sp_{2n}(\bbA)$. Here $\delta_{P_k}$ is the modular character of $P_k$. For a holomorphic section $f_s$ of $I_n(s)$, let $E_{f_s}$ be the Eisenstein series in (\ref{eis1}). The inducing data are not cuspidal, hence the term ``degenerate.'' The poles of $E_{f_s}$ coincide with those of the constant term along $B$,
\[
  C_0(g,f_s)=\int_{U(F)\backslash U(\bbA)}E_{f_s}(ug)\,du.
\]
This constant term is computed by unfolding; see \cite{Muic08}.

For $t=\diag(a_1,\ldots,a_n,a_n^{-1},\ldots,a_1^{-1})\in T(\bbA)$, define the character $\Lambda_s$ of $T(F)\backslash T(\bbA)$ by
\[
  \Lambda_s(t) = \chi\!\Big(\prod_{i=1}^k a_i\Big)\, \Big(|a_1|^{\,s-\frac{k-1}{2}}\, |a_2|^{\,s-\frac{k-3}{2}}\cdots |a_k|^{\,s+\frac{k-1}{2}}\Big)\, \Big(|a_{k+1}|^{\,k-n}\, |a_{k+2}|^{\,k-n+1}\cdots |a_n|^{-1}\Big).
\]
The induced representation
\[
  I(\Lambda_s)=\Ind_{B(\bbA)}^{\Sp_{2n}(\bbA)}(\Lambda_s)
\]
contains $I_n(s)$ as a subrepresentation.

For $w\in W_n$, write $\Lambda_s^w$ for the character obtained from the action of $w$ on $T(\bbA)$. Given parabolic subgroups $P$ and $Q$, let $W(P,Q)$ consist of the elements $w\in W_n$ such that $wM_Pw^{-1}=M_Q$ and $w$ is the representative of minimal length in $W(M_Q)wW(M_P)$. For $w\in W(P,Q)$, the global intertwining operator $M(w,s):I(\Lambda_s)\to I(\Lambda_s^w)$ is defined by meromorphic continuation of
\[
  M(w,s) f_s(g) = \int\limits_{(U(\bbA)\cap wU(\bbA)w^{-1})\setminus U(\bbA) }f_s(w^{-1}ug)\, du,
\]
from its domain of convergence; see \cite[Chapter~2]{BL24}. The operator does not depend on the representative of $w$. For a factorizable section $f=\otimes_vf_v$, it factors as $M(w,s)f=\otimes_vM_v(w,s)f_v$.

\subsection{Unipotent orbits and Fourier coefficients}\label{sec204}
Unipotent orbits of $\Sp_{2n}$ are parameterized by partitions of $2n$ in which each odd part has even multiplicity. Write $\calO=(p_1^{r_1}\cdots p_z^{r_z})$, where $p_1>\cdots>p_z$ and $r_i>0$. For each $p_i$, take $r_i$ copies of $h_{p_i}(t)=\diag(t^{p_i-1},t^{p_i-3},\ldots,t^{3-p_i},t^{1-p_i})$. Combining these diagonal entries and arranging their exponents in nonincreasing order gives a torus $h_\calO(t)$ parametrized by $t$; see \cite{GIN06}.

Define
\[
  \begin{aligned}
    V_{i,\calO}(F)&:=\bigl\langle X_\alpha(r):\alpha>0,\ r\in F, \\
    &\qquad h_\calO(t)X_\alpha(r)h_\calO(t)^{-1}=X_\alpha(t^jr)\text{ for some }j\geq i\bigr\rangle, \\
    M(\calO)(F)&:=T\cdot\bigl\langle X_\alpha(r):\alpha\in\Sigma,\ r\in F, \\
    &\qquad h_\calO(t)X_\alpha(r)h_\calO(t)^{-1}=X_\alpha(r)\bigr\rangle.
  \end{aligned}
\]
Let $V_{2,\calO}^{(1)}$ be the commutator subgroup of $V_{2,\calO}$, and put $L_{2,\calO}=V_{2,\calO}/V_{2,\calO}^{(1)}$. Choose a representative $u_0$ of the open $M(\calO)$-orbit in $L_{2,\calO}$, and put $M^{u_0}(\calO)=\operatorname{Stab}_{M(\calO)}(u_0)$. We call $\psi_{\calO}: L_{2,\calO}(F)\setminus L_{2,\calO}(\bbA) \rightarrow \bbC^\times$ a generic (automorphic) character if the connected component of its stabilizer in $M(\calO)(F)$ has the same group type as $M^{u_0}(\calO)(\bar{F})$.

Let $(\pi,\calV)$ be an automorphic representation of $\Sp_{2n}(\bbA)$ and $\psi_\calO:V_{2,\calO}(F)\setminus V_{2,\calO}(\bbA) \rightarrow \bbC^\times $ be a generic character associated with a unipotent orbit $\calO$ in $\Sp_{2n}$. We define the Fourier coefficient of any automorphic function $\varphi \in \calV$ associated with $\calO$ by
\[
  F_{\psi_\calO}(\varphi)(g) = \int\limits_{V_{2,\calO}(F)\setminus V_{2,\calO}(\bbA)}\varphi(ug)\psi_\calO(u) \,du.
\]

We say that $\calO$ supports $\pi$ if this integral is nonzero for some $\varphi\in\calV$ and some generic $\psi_\calO$.

\section{Fourier coefficients of the degenerate Eisenstein series}
Assume $3 \leqslant k < n$ for the rest of the paper. By \cite[Chapter 5]{GIN06}, the Eisenstein series $E_{f_s}(g)$ is supported on the unipotent orbit $\calO=(3^21^{2n-6})$. For this orbit, the Fourier coefficient vanishes on $P_1$ but not on $P_2$, and $V_{2,\calO}=U_2$. Define $\psi_\calO:U_2(F)\setminus U_2(\bbA)\to\bbC^\times$ by $\psi_\calO(u)=\psi(u_{1,3}+u_{2,2n-2})$. We study the Fourier coefficient $F_{\psi_\calO}(E_{f_s})(g)$ defined in \eqref{fc}. Write $M_{\psi_\calO}(\bbA)$ for the stabilizer of $\psi_\calO$ in $M_2(\bbA)$. The embedding $\triangle$ in \eqref{triangle} identifies this stabilizer with $\Sp_2(\bbA)\times\Sp_{2(n-3)}(\bbA)$.

Let $Q_{\psi_\calO} = M_{\psi_\calO}\ltimes U_2$. By unfolding the Eisenstein series, we rewrite the Fourier coefficient as
\begin{equation}\label{int400}
  \begin{split}
    F_{\psi_\calO}(E_{f_s})(g) &= \int\limits_{U_2(F)\setminus U_2(\bbA)} \sum_{\gamma\in P_k(F)\setminus \Sp_{2n}(F)}f_s(\gamma ug)\psi_\calO(u)\, du \\
    &= \sum_{\gamma \in P_k(F)\setminus \Sp_{2n}(F)/Q_{\psi_\calO}(F)} \int\limits_{U_2(F)\setminus U_2(\bbA)} \sum_{\beta \in Q_{\psi_\calO}^\gamma(F)\setminus Q_{\psi_\calO}(F)} f_s(\gamma\beta u g )\psi_\calO(u)\, du,
  \end{split}
\end{equation}
where $Q_{\psi_\calO}^\gamma(F)=\{\beta\in Q_{\psi_\calO}(F):\gamma\beta\gamma^{-1}\in P_k(F)\}$.

\begin{lemma}
There are unipotent subgroups $V_w\subset U\cap M_2$, indexed by $w\in W_n$, such that
\[
  \Sp_{2n}(F) = \bigcup_{w\in W_n} P_k(F)\,w\,V_w(F)\,Q_{\psi_\calO}(F).
\]
\end{lemma}
\begin{proof}
Since $U_2 \subset Q_{\psi_\calO}$, we have a surjection
\[
  P_k(F) \backslash \Sp_{2n}(F) / U_2(F) \twoheadrightarrow P_k(F) \backslash \Sp_{2n}(F) / Q_{\psi_\calO}(F).
\]
By the Bruhat decomposition, every $g\in\Sp_{2n}(F)$ has the form $g=b_1wb_2$, with $b_1,b_2\in B(F)$ and $w\in W_n$. Since $B\subset P_k$, the factor $b_1$ is absorbed into $P_k$ on the left. Write $b_2=tu$, with $t\in T(F)$ and $u\in U(F)$. The identity $wt=(wtw^{-1})w$ and the inclusion $wtw^{-1}\in T(F)\subset P_k(F)$ allow us to absorb $t$ as well. Hence
\[
  P_k(F) \backslash \Sp_{2n}(F) / U(F) = \bigsqcup_{w \in W(P_k) \backslash W_n} P_k(F)\, w\, U(F) / U(F).
\]

The decomposition $U=(U\cap M_2)\cdot U_2$ gives $U/U_2\cong U\cap M_2=\prod_{\alpha\in\Sigma^+(M_2)}X_\alpha$. For $\alpha\in\Sigma^+(M_2)$, we have $wX_\alpha(r)w^{-1}=X_{w\alpha}(\pm r)$. If $w\alpha\in\Sigma(P_k)$, this root subgroup can be absorbed into $P_k$ on the left. The right quotient by $Q_{\psi_\calO}=M_{\psi_\calO}\ltimes U_2$ also removes the root subgroups contained in $M_{\psi_\calO}$. The remaining root subgroups are indexed by
\[
  \calR_w := \{\alpha \in \Sigma^+(M_2) : X_\alpha \not\subset M_{\psi_\calO} \text{ and } w \alpha \notin \Sigma(P_k)\},
\]
and we set $V_w := \prod_{\alpha \in \calR_w} X_\alpha$. Every double coset in $P_k(F) \backslash \Sp_{2n}(F) / Q_{\psi_\calO}(F)$ admits a representative of the form $\gamma = w v_w$ with $w \in W_n$ and $v_w \in V_w(F)$.
\end{proof}

Set $\tau_1=e_1-e_3$ and $\tau_2=e_2+e_3$. These are the two positive roots whose root subgroups carry the nontrivial restrictions of $\psi_\calO$.
\begin{lemma}
The summand in \eqref{int400} corresponding to $wv_w$ vanishes unless $w(1),w(2)\in\{-1,\dots,-k\}$ and $w(3)\in\pm\{k+1,\dots,n\}$.
\end{lemma}
\begin{proof}
Fix a summand corresponding to $\gamma=wv_w$. Each $\beta\in Q_{\psi_\calO}(F)\subset P_2(F)$ normalizes $U_2(F)$ and stabilizes $\psi_\calO$. Thus the change of variables $u\mapsto\beta^{-1}u\beta$ on $U_2(F)\backslash U_2(\bbA)$ preserves the Haar measure and the character. It gives
\begin{equation}\label{beta}
  \int_{U_2(F) \backslash U_2(\bbA)} f_s(\gamma \beta u g)\, \psi_\calO(u)\, du = \int_{U_2(F) \backslash U_2(\bbA)} f_s(\gamma u \beta g)\, \psi_\calO(u)\, du.
\end{equation}

For $i=1,2$, put $Y_{i,v}(r)=v_w^{-1}X_{\tau_i}(r)v_w$. This defines a closed one-parameter subgroup of $U_2(\bbA)$, with $Y_{i,v}(F)=Y_{i,v}(\bbA)\cap U_2(F)$. The $U_2$-integration therefore decomposes along $Y_{i,v}(\bbA)$:
\[
  U_2(F) \backslash U_2(\bbA) \;\simeq\; \bigl(U_2(F)\, Y_{i, v}(\bbA) \backslash U_2(\bbA)\bigr) \times (F \backslash \bbA), \qquad du = dr \, du'.
\]
A direct computation shows that $v_w^{-1}X_{\tau_i}(r)v_w$ differs from $X_{\tau_i}(r)$ only by root subgroups on which $\psi_\calO$ is trivial. Hence $\psi_\calO(Y_{i,v}(r)u')=\psi(r)\psi_\calO(u')$. Substituting into \eqref{beta} gives
\begin{equation}\label{inner}
  \int_{U_2(F) Y_{i, v}(\bbA) \backslash U_2(\bbA)} \left( \int_{F \backslash \bbA} f_s\bigl(\gamma\, Y_{i, v}(r)\, u' \beta g\bigr)\, \psi(r)\, dr \right) \psi_\calO(u')\, du'.
\end{equation}

We have $\gamma=wv_w$ and $wX_{\tau_i}(r)w^{-1}=X_{w\tau_i}(\pm r)$. After replacing $r$ by $-r$ if necessary, we obtain
\[
  f_s\bigl(\gamma Y_{i,v}(r)u'\beta g\bigr)=f_s\bigl(X_{w\tau_i}(r)\gamma u'\beta g\bigr).
\]
Thus the vanishing argument for \eqref{inner} depends only on the position of $X_{w\tau_i}$.

Suppose $X_{w\tau_i}\subset P_k$. The normalized inducing character is trivial on this unipotent root subgroup. The left $P_k(\bbA)^1$-invariance of $f_s$ therefore removes $X_{w\tau_i}(r)$ from the integrand. The inner integral becomes $\int_{F\backslash\bbA}\psi(r)\,dr=0$. Hence the summand can be nonzero only if
\begin{equation}\label{roots}
  X_{w \tau_1} \not\subset P_k \quad \text{and} \quad X_{w \tau_2} \not\subset P_k.
\end{equation}
In the signed permutation model for $W_n$, condition \eqref{roots} is equivalent to
\begin{equation}\label{w123}
  w(1), w(2) \in \{-1, \ldots, -k\}, \qquad w(3) \in \pm\{k+1, \ldots, n\}.
\end{equation}
Indeed, if $w(3) \notin \pm\{k+1, \ldots, n\}$ or one of $w(1), w(2)$ lies in $\pm\{k+1, \ldots, n\}$, then $w \tau_1$ or $w \tau_2$ lies in $\Sigma(P_k)$, contradicting \eqref{roots}.
\end{proof}

Let $\calS$ be the set of positive roots $\alpha$ of $M_2$ such that $X_\alpha\not\subset M_{\psi_\calO}$. Fix the Weyl element
\begin{equation}\label{w}
  w = \begin{pmatrix}
    & & I_{k-2} & & & & \\
    & & & & & & I_{2} \\
    & 1 & & & & & \\
    & & & I_{2(n-k-1)} & & & \\
    & & & & & 1 & \\
    -I_{2} & & & & & & \\
    & & & & I_{k-2} & &
  \end{pmatrix} .
\end{equation}
This $w$ satisfies \eqref{w123}: the $-I_2$ block sends $e_1$ to $-e_k$ and $e_2$ to $-e_{k-1}$. The upper $1$ block sends $e_3$ to $e_{k+1}$. The remaining action is $e_j \mapsto e_{j-3}$ for $4 \leq j \leq k+1$ and $e_j \mapsto e_j$ for $j \geq k+2$.

If $\alpha=e_3-e_j$ with $4\leq j\leq k+1$, then $w\alpha=e_{k+1}-e_{j-3}$. Since $j-3\in\{1,\ldots,k-2\}$, this root lies outside $\Sigma(P_k)$. Thus $\alpha\in\calR_w$.

If $\alpha=e_3-e_j$ with $j\geq k+2$, both indices of $w\alpha=e_{k+1}-e_j$ lie in $\{k+1,\ldots,n\}$. Hence $w\alpha\in\Sigma(\Sp_{2(n-k)})\subset\Sigma(P_k)$, so $\alpha\notin\calR_w$. If $\alpha=e_3+e_j$ with $j\geq4$, then $w\alpha\in\Sigma(P_k)$, so again $\alpha\notin\calR_w$. We obtain $\calR_w=\{e_3-e_j:4\leq j\leq k+1\}$ and
\begin{equation}\label{vw}
  V_w = \left\{\left. \begin{pmatrix}
    I_2 & & & & & & \\
    & 1 & v & & & & \\
    & & I_{k-2} & & & & \\
    & & & I_{2(n-k-1)} & & & \\
    & & & & I_{k-2} & v^* & \\
    & & & & & 1 & \\
    & & & & & & I_2
  \end{pmatrix} \in \Sp_{2n}(F) \,\right|\, v \in F^{k-2}\right\}.
\end{equation}

\begin{lemma}
Every family $w'V_{w'}$ with $w'$ satisfying \eqref{w123} is equivalent to $wV_w$ under left multiplication by $W(P_k)$ and right multiplication by $W(Q_{\psi_\calO})$.
\end{lemma}
\begin{proof}
We construct $\sigma_\ell\in W(P_k)$ and $\sigma_r\in1\times W_{n-3}\subset W(Q_{\psi_\calO})$ such that
\begin{equation}\label{families}
  \sigma_\ell \cdot w V_w \cdot \sigma_r = w' V_{w'}.
\end{equation}

Recall that $W(P_k)\cong S_k\times W(\Sp_{2(n-k)})$. The factor $S_k$ permutes $\{1,\ldots,k\}$, and $W(\Sp_{2(n-k)})$ acts by signed permutations on $\pm\{k+1,\ldots,n\}$. The factor $W_{n-3}$ in $W(Q_{\psi_\calO})$ acts by signed permutations on $\{4,\ldots,n\}$ and fixes $\{1,2,3\}$.

The values $w(j),w'(j)$ lie in $\{-1,\ldots,-k\}$ for $j=1,2$. Since $S_k$ acts transitively on ordered pairs of distinct elements of this set, choose $\tau\in S_k$ such that $\tau(w(j))=w'(j)$ for $j=1,2$. Likewise, $w(3),w'(3)\in\pm\{k+1,\ldots,n\}$. Transitivity of $W(\Sp_{2(n-k)})$ on this set gives $\rho\in W(\Sp_{2(n-k)})$ such that $\rho(w(3))=w'(3)$. Put $\sigma_\ell=(\tau,\rho)\in W(P_k)$. Then $\sigma_\ell w$ agrees with $w'$ on $\{1,2,3\}$. Consequently, the two elements map $\{\pm4,\ldots,\pm n\}$ to the same subset of $\{\pm1,\ldots,\pm n\}$.

For $j \in \{4, \ldots, n\}$, define $\sigma_r(j) := w^{-1}\bigl(\sigma_\ell^{-1}(w'(j))\bigr)$. The argument $\sigma_\ell^{-1}(w'(j))$ lies in $w(\{\pm 4, \ldots, \pm n\})$, so $\sigma_r(j) \in \{\pm 4, \ldots, \pm n\}$ and $\sigma_r$ is a signed permutation of $\{4, \ldots, n\}$. Hence $\sigma_r \in 1 \times W_{n-3}$. Direct verification gives $w' = \sigma_\ell w \sigma_r$.

It remains to verify
\begin{equation}\label{Vw}
  V_{w'} = \sigma_r^{-1} V_w \sigma_r,
\end{equation}
which, together with $w'=\sigma_\ell w\sigma_r$, gives \eqref{families}.

For $\sigma\in W(P_k)$ and $w\in W_n$, we have $V_{\sigma w}=V_w$. Indeed, $W(P_k)=W(M_k)$ stabilizes $\Sigma(M_k)$. Applying $(\tau,\rho)\in S_k\times W(\Sp_{2(n-k)})$ to $e_a\pm e_b\in\Sigma(U_k)$ with $a\leq k<b$ gives $e_{\tau(a)}\pm\epsilon e_{\rho'(b)}\in\Sigma(U_k)$ for either sign $\epsilon$. Thus $W(P_k)$ preserves $\Sigma(P_k)$. The condition $w\alpha\notin\Sigma(P_k)$ is unchanged by left multiplication of $w$ by $\sigma$. It follows that $\calR_{\sigma w}=\calR_w$ and $V_{\sigma w}=V_w$.

For $\sigma_r\in1\times W_{n-3}$ and $w\in W_n$, we have $V_{w\sigma_r}=\sigma_r^{-1}V_w\sigma_r$. The element $\sigma_r$ fixes $e_3$ and acts by signed permutations on $\{e_4,\ldots,e_n\}$. Hence it preserves $\calS$. For $\beta\in\calS$, we have
\[
  \beta \in \calR_{w \sigma_r} \iff w \sigma_r(\beta) \notin \Sigma(P_k) \iff \sigma_r(\beta) \in \calR_w \iff \beta \in \sigma_r^{-1}(\calR_w),
\]
giving $\calR_{w \sigma_r} = \sigma_r^{-1}(\calR_w)$. The subgroup identity follows from the standard formula $\sigma_r^{-1} X_\alpha \sigma_r = X_{\sigma_r^{-1}(\alpha)}$:
\[
  \sigma_r^{-1} V_w \sigma_r = \prod_{\alpha \in \calR_w} X_{\sigma_r^{-1}(\alpha)} = \prod_{\beta \in \sigma_r^{-1}(\calR_w)} X_\beta = \prod_{\beta \in \calR_{w \sigma_r}} X_\beta = V_{w \sigma_r}.
\]

Since $w'=\sigma_\ell w\sigma_r$, these identities give $V_{w'}=V_{\sigma_\ell w\sigma_r}=V_{w\sigma_r}=\sigma_r^{-1}V_w\sigma_r$. This proves \eqref{Vw} and hence \eqref{families}. Thus $wV_w$ represents the unique class of families that can contribute nontrivially under these Weyl group actions.
\end{proof}

\begin{proposition}
Let $H=U_2^w:=U_2\cap w^{-1}P_kw$. Then
\[
  \begin{aligned}
    F_{\psi_\calO}(E_{f_s})(g) &= \sum_{\beta' \in P_{k-2,n-3}(F)\backslash\Sp_{2(n-3)}(F)} \int_{H(F)\backslash U_2(\mathbb A)} f_s\bigl(wu\iota_c(\beta')g\bigr)\psi_{\calO}(u)\,du \\
    &= \sum_{\beta' \in P_{k-2,n-3}(F)\backslash\Sp_{2(n-3)}(F)} \int_{H(\mathbb A)\backslash U_2(\mathbb A)} \int_{H(F)\backslash H(\mathbb A)} f_s\bigl(whu\iota_c(\beta')g\bigr) \psi_{\calO}(hu)\,dh\,d\dot u \\
    &= \sum_{\beta' \in P_{k-2,n-3}(F)\backslash\Sp_{2(n-3)}(F)} \int_{H(\mathbb A)\backslash U_2(\mathbb A)} f_s\bigl(wu\iota_c(\beta')g\bigr) \psi_{\calO}(u)\,d\dot u.
  \end{aligned}
\]
The last equality uses $wH(\mathbb A)w^{-1}\subset(U\cap M_k)(\mathbb A)\subset P_k(\mathbb A)$ and the normalization $\operatorname{vol}(H(F)\backslash H(\mathbb A))=1$.
\end{proposition}
\begin{proof}
\newcommand{\nTwoPlus}{n_{2,+}^{k+1}}

Fix $v=(v_1,\dots,v_{k-2})\in F^{k-2}$ and embed it in $V_w$ via \eqref{vw}. Let $r=(r_1,\dots,r_{k-2})\in \bbA^{k-2}$. Define the $(k-1)$-tuple
\begin{equation}\label{rtuple}
  \bigl(\delta(r), r\bigr) := \bigl(-\sum_{i=1}^{k-2}r_iv_i,\, r_1,\, r_2,\, \dots,\, r_{k-2}\bigr) \in \bbA^{k-1},
\end{equation}
where $\delta(r)=-\sum_{i=1}^{k-2}r_iv_i$. For $(a_3,\dots,a_{k+1})\in\bbA^{k-1}$, define $\nTwoPlus((a_3,\dots,a_{k+1}))=\prod_{j=3}^{k+1}X_{e_2+e_j}(a_j)$.

By construction, $n_{2,+}^{k+1}((\delta(r),r))\in U_2(\bbA)$. Write $u=n_{2,+}^{k+1}((\delta(r),r))u'$ with $u'\in U_2(\bbA)$. To move the first factor to the left, we compute its conjugate by $v$. In the unipotent coordinates, this gives
\[
  v\,n_{2,+}^{k+1}((\delta(r),r))\,v^{-1}=n_{2,+}^{k+1}\!\bigl(\,\delta(r)+\sum_{i=1}^{k-2}r_i v_i,r\bigr).
\]
By the definition of $\delta(r)$ in \eqref{rtuple}, this becomes $v\,n_{2,+}^{k+1}((\delta(r),r))\,v^{-1}=n_{2,+}^{k+1}((0,r))$. Consequently,
\[
  f_s\bigl(wv\,n_{2,+}^{k+1}((\delta(r),r))\,u'\beta g\bigr)=f_s\bigl(w\,n_{2,+}^{k+1}(0,r)\,v\,u'\beta g\bigr).
\]

Since $wn_{2,+}^{k+1}(0,r)w^{-1}\in P_k(\bbA)$, the left $P_k(\bbA)^1$-invariance of $f_s$ gives $f_s(wn_{2,+}^{k+1}(0,r)vu'\beta g)=f_s(wvu'\beta g)$. The result is independent of $r$. The character contributes
\[
  \psi_\calO\bigl(n_{2,+}^{k+1}((\delta(r),r))\bigr) =\psi\!\left(-\sum_{i=1}^{k-2}r_i v_i\right) =\prod_{i=1}^{k-2}\psi(-r_i v_i).
\]
Thus, for any fixed $v\in V_w$, the Fourier coefficient contains the inner integral
\[
  \int_{(F\backslash\bbA)^{k-2}} \left(\prod_{i=1}^{k-2}\psi(-r_i v_i)\right)\,\prod_{i=1}^{k-2}dr_i.
\]
By orthogonality of additive characters, this integral is nonzero only if $v_i=0$ for all $i=1,\dots,k-2$. Thus only $v=0$ contributes to the sum over $V_w(F)$. Since $Q_{\psi_\calO}^w$ is independent of $v$, we obtain
\begin{equation}\label{vzero}
  F_{\psi_\calO}(E_{f_s})(g) =\int_{U_2(F)\setminus U_2(\bbA)} \sum_{\beta\in Q_{\psi_\calO}^w(F)\setminus Q_{\psi_\calO}(F)} f_s(w u\beta g)\,\psi_\calO(u)\,du.
\end{equation}

Furthermore, since $M^w_{\psi_{\mathcal O}}\simeq \Sp_2\times P_{k-2,n-3}$, every class $\beta\in Q^w_{\psi_{\mathcal O}}(F)\backslash Q_{\psi_{\mathcal O}}(F)$ may be represented by a pair
\[
  \beta'\in P_{k-2,n-3}(F)\backslash \Sp_{2(n-3)}(F), \qquad u_0\in U_2^w(F)\backslash U_2(F),
\]
corresponding to its Levi and unipotent components, respectively. Substitute this into \eqref{vzero}. Since $\psi_{\mathcal O}$ is trivial on $U_2^w(\bbA)$ and $U_{2,w}(\bbA)\simeq U_2^w(\bbA)\backslash U_2(\bbA)$, the summation over $U_2^w(F)\backslash U_2(F)$ can be absorbed into the integration over $U_2(F)\backslash U_2(\bbA)$. This gives the asserted formula.
\end{proof}

The Fourier coefficient $F_{\psi_\calO}(E_{f_s})(g)$ is left invariant under $M_{\psi_\calO}(F)$. Its restriction to $\Sp_{2(n-3)}(\bbA)$ is therefore left invariant under $\Sp_{2(n-3)}(F)$.
\begin{theorem}\label{thm401}
Let $E_{f_s}(g)$ be an Eisenstein series associated with
\[
  I_n(s)=\Ind_{P_k(\bbA)}^{\Sp_{2n}(\bbA)}\left(\chi|\det|^s\otimes\id_{\Sp_{2(n-k)}}\right).
\]
Fix a generic character $\psi_\calO$. For $\Re(s)$ sufficiently large, the restriction of $F_{\psi_\calO}(E_{f_s})(g)$ to $\Sp_{2(n-3)}(\bbA)$ is an Eisenstein series associated with
\[
  I_{n-3}(s) = \Ind_{P_{k-2,n-3}(\bbA)}^{\Sp_{2(n-3)}(\bbA)} \chi|\det|^{s}\otimes \id_{\Sp_{2(n-k-1)}}.
\]
The corresponding section admits meromorphic continuation to the whole $s$-plane. Let $s_0$ satisfy $\Re(s_0)\geq0$, and suppose the continued local descent operator is regular at $s_0$. The local map at $s_0$ is surjective at non-Archimedean places and has dense image at Archimedean places.
\end{theorem}

The following lemma reduces Theorem~\ref{thm401} to the local calculations in the next section. We complete the proof of the theorem in Section~\ref{last}.

\begin{lemma}
Let $\iota_c$ denote the embedding of $\Sp_{2(n-3)}$ into the central block of $\Sp_{2n}$. For $g\in\Sp_{2(n-3)}(\bbA)$, define
\[
  \calF(f_s,g) \;:=\; \int_{U_{2,w}(\bbA)} f_s\!\bigl(wu\iota_c(g)\bigr)\,\psi_\calO(u)\,du.
\]
Then $\calF(f_s,\cdot)$ defines a meromorphic section of $I_{n-3}(s)$. Moreover, if $f_s=\prod_v f_{s,v}$ is factorizable, then
\begin{equation}\label{int407}
  \calF(f_s,g) \;=\; \prod_v \int_{U_{2,w}(F_v)} f_{s,v}\!\bigl(wu_v\iota_c(g_v)\bigr)\,\psi_{\calO,v}(u_v)\,du_v,
\end{equation}
where $\psi_\calO=\prod_v \psi_{\calO,v}$ and $g=(g_v)_v$.
\end{lemma}
\begin{proof}
We have
\[
  F_{\psi_\calO}(E_{f_s})(g) = \sum_{\beta \in P_{k-2,n-3}(F)\backslash \Sp_{2(n-3)}(F)} \calF(f_s,\beta g), \qquad g\in \Sp_{2(n-3)}(\bbA).
\]
We verify the $P_{k-2,n-3}(\bbA)$-equivariance required for $I_{n-3}(s)$.

Put $P'=P_{k-2,n-3}$. Write $p'=(h_1,h_2)v\in P'(\bbA)$, with $(h_1,h_2)\in\GL_{k-2}(\bbA)\times\Sp_{2(n-k-1)}(\bbA)$ and $v\in U_{P'}(\bbA)$. The Weyl element in \eqref{w} satisfies $w(e_j)=e_{j-3}$ for $4\leq j\leq k+1$. Under $\iota_c$, the $\GL_{k-2}$-block acts on $\langle e_4,\ldots,e_{k+1}\rangle$. Conjugation by $w$ therefore carries it into the $\GL_k$-block of $P_k$.

Set $p_w(p')=w\iota_c(p')w^{-1}\in P_k(\bbA)$, and write its Levi component as $m_k(p')=(a(p'),b(p'))\in\GL_k(\bbA)\times\Sp_{2(n-k)}(\bbA)$. Then $a(p')=\diag(h_1,I_2)$ up to multiplication by unipotent factors. Hence
\begin{equation}\label{det}
  \det a(p')=\det h_1.
\end{equation}
Conjugation by $w$ carries $\iota_c(v)$ into $U_k$. It carries the $\Sp_{2(n-k-1)}$-factor into the $\Sp_{2(n-k)}$-factor of the Levi of $P_k$.

We compute
\[
  \begin{aligned}
    \calF(f_s,p'g) &= \int_{U_{2,w}(\bbA)} f_s\!\bigl(wu\,\iota_{c}(p'g)\bigr)\psi_\calO(u)\,du \\
    &= \int_{U_{2,w}(\bbA)} f_s\!\bigl(wu\,\iota_{c}(p')\iota_{c}(g)\bigr) \psi_\calO(u)\,du.
  \end{aligned}
\]

Since $w\iota_{c}(p')=p_w(p')w$, we have
\[
  wu\iota_{c}(p') = w\iota_{c}(p') \left(\iota_{c}(p')^{-1}u\iota_{c}(p')\right) = p_w(p')w \left(\iota_{c}(p')^{-1}u\iota_{c}(p')\right).
\]
Under the central embedding, $\iota_{c}(p')$ preserves the quotient $U_{2,w}=U_2^w\backslash U_2$, and hence induces a change of variables on $U_{2,w}$. Put $u'=\iota_{c}(p')^{-1}u\iota_{c}(p')$. The character $\psi_\calO$ is preserved under this change of variables. Thus, by the $P_k(\bbA)$-equivariance of $f_s$,
\[
  \begin{aligned}
    \calF(f_s,p'g) &= \chi(\det h_1)|\det h_1|^s \delta_{P_k}^{1/2}(p_w(p')) \\
    &\qquad\qquad\cdot \int_{U_{2,w}(\bbA)} f_s(wu'\iota_{c}(g))\psi_\calO(u')\, d\!\left(\iota_{c}(p')u'\iota_{c}(p')^{-1}\right).
  \end{aligned}
\]

In the coordinates $u(x,y,z)$ for $U_{2,w}$, the change of variables sends $y\in\Mat_{2\times(k-2)}(\bbA)$ to $yh_1^{-1}$. The $x$- and $z$-coordinates undergo affine transformations with determinant one and translations preserving the character. Since $y$ has two rows, $d(yh_1^{-1})=|\det h_1|^{-2}dy$. Thus $du=|\det h_1|^{-2}du'$, and
\[
  \calF(f_s,p'g) = \chi(\det h_1)|\det h_1|^s \delta_{P_k}^{1/2}(p_w(p')) |\det h_1|^{-2} \calF(f_s,g).
\]

The Levi subgroup of $P_k\subset\Sp_{2n}$ is $\GL_k\times\Sp_{2(n-k)}$, so $\delta_{P_k}(a,h)=|\det a|^{2n-k+1}$. By \eqref{det}, $\delta_{P_k}^{1/2}(p_w(p'))=|\det h_1|^{(2n-k+1)/2}$. The Levi subgroup of $P'=P_{k-2,n-3}\subset\Sp_{2(n-3)}$ is $\GL_{k-2}\times\Sp_{2(n-k-1)}$, so
\[
  \delta_{P'}(h_1,h_2) = |\det h_1|^{2(n-3)-(k-2)+1} = |\det h_1|^{2n-k-3}.
\]
Thus
\[
  \delta_{P_k}^{1/2}(p_w(p'))|\det h_1|^{-2} = |\det h_1|^{(2n-k+1)/2-2} = |\det h_1|^{(2n-k-3)/2} = \delta_{P'}^{1/2}(h_1,h_2).
\]
Therefore
\[
  \calF(f_s,p'g)=\chi(\det h_1)|\det h_1|^s\delta_{P'}^{1/2}(h_1,h_2)\calF(f_s,g).
\]

Smoothness and right $K$-finiteness follow from the corresponding properties of $f_s$. The local computations in the next section give convergence for $\Re(s)\gg0$ and meromorphic continuation. Hence $\calF(f_s,\cdot)$ is a meromorphic section of $I_{n-3}(s)$.

Finally, assume $f_s=\prod_v f_{s,v}$ and $\psi_\calO=\prod_v\psi_{\calO,v}$ are factorizable. With the product Haar measure on $U_{2,w}(\bbA)$, Fubini's theorem gives the Euler factorization \eqref{int407} in the initial domain of absolute convergence. The identity then extends meromorphically in $s$.
\end{proof}
\section{Local computations}
At a non-Archimedean place $v$, let $\calO_v$ be the ring of integers of $F_v$. Choose a uniformizer $\varpi$ with $|\varpi|_v=q^{-1}$.

\subsection{The unramified case}

Assume that $f_{s,v}$ is a spherical local holomorphic section of
\[
  I_{n,v}(s) = \Ind_{P_k(F_v)}^{\Sp_{2n}(F_v)}\chi_v|\det|^s\otimes \id_{\Sp_{2(n-k)}}.
\]
Normalize the local additive character $\psi_v$ to be trivial on $\calO_v$ and nontrivial on $\varpi^{-1}\calO_v$.

Set $K_v=\Sp_{2n}(\calO_v)$. Move $w$ to the right and make the change of variables $u\mapsto wuw^{-1}$. The right $K_v$-invariance of $f_{s,v}$ gives
\begin{equation}\label{int501}
  \begin{split}
    \calF_v(f_{s,v},g) &= \int\limits_{U_{2,w}(F_v)}f_{s,v}(w u\iota_{c}(g))\psi_{\calO,v}(u)\,du \\
    &=\int\limits_{\overline{U}_{2,w}(F_v)}f_{s,v}(u\iota(g))\psi^{w^{-1}}_{\calO,v}(u)\,du, \quad g\in \Sp_{2(n-3)}(F_v),
  \end{split}
\end{equation}
where $\overline{U}_{2,w}=wU_{2,w}w^{-1}$. In matrix form,
\begin{equation}\label{def501}
  \overline{U}_{2,w} = w U_{2,w}w^{-1} = \left.\left\{ \begin{pmatrix}
    I_{k-2}&&&& \\
    &I_2&&& \\
    0&x&I_{2(n-k)}&& \\
    y&z&x^\ast&I_2& \\
    0&y^\ast&0&&I_{k-2}
  \end{pmatrix} \right\vert x = \begin{pmatrix}
    {}^tx_4 \\
    {}^tx_3 \\
    {}^tx_2 \\
    {}^tx_1
  \end{pmatrix} J_2 \right\},
\end{equation}
where $x_2,x_3\in \Mat_{2\times (n-k-1)}$ and $x_1,x_4\in \Mat_{2\times 1}$. Write $\overline{u(x,y,z)}$ for the element in (\ref{def501}). The character is $\psi_{\calO,v}^{w^{-1}}(u)=\psi_{\calO,v}(w^{-1}uw)$.

\begin{theorem}\label{thm501}
Let $f_{s,v}$ be a spherical local holomorphic section in $I_{n,v}(s)$. The local Fourier coefficient
\[
  \calF_v(f_{s,v},g) = \int\limits_{\overline{U}_{2,w}(F_v)}f_{s,v}(u\iota(g))\psi^{w^{-1}}_{\calO,v}(u)\,du
\]
defines a local holomorphic section in
\[
  I_{n-3,v}(s) = \Ind_{P_{k-2,n-3}(F_v)}^{\Sp_{2(n-3)}(F_v)} \chi_v|\det|^{s}\otimes \id_{\Sp_{2(n-k-1)}}.
\]
For the normalized spherical vectors $f^\circ$, we have
\[
  \begin{aligned}
    \calF_v(f^{\circ}_{n,v,s},g) &=L_v^{-1}(\chi_v^2,2s+k-1)\, L_v^{-1}\!\left(\chi_v,s-\frac{k+1}{2}+n\right) \\
    &\quad\times L_v^{-1}\!\left(\chi_v,s-\frac{k-1}{2}+n\right)\, f^{\circ}_{n-3,v,s}(g).
  \end{aligned}
\]
\end{theorem}

\begin{proof}
The local transformation formula gives $\calF_v(f_{s,v},\cdot)\in I_{n-3,v}(s)$. It remains to prove holomorphy. By linearity, it suffices to take $f_{s,v}=f^\circ_{n,v,s}$. Since the descended section is spherical, it suffices to compute its value at $g=\id$.

Let $x=(x_{i,j})\in \Mat_{2(n-k)\times 2}(F_v)$. Its image under (\ref{def501}) is $\overline{u(x,0,0)}$. These coordinates correspond to the roots $\{-e_i\pm e_j\mid k-1\le i\le k<j\le n\}$. For $u=\overline{u(x,y,z)}$, the character is $\psi_{\calO,v}^{w^{-1}}(u)=\psi_v(x_{1,1}-x_{2(n-k),2})$. Write
\[
  y = \begin{pmatrix}
    y_1 \\
    y_2
  \end{pmatrix} \in\Mat_{2\times (k-2)}(F_v), \quad y_i = \begin{pmatrix}
    y_{i,1}&y_{i,2}&\cdots&y_{i,(k-2)}
  \end{pmatrix} ,\; i =1,2,
\]
and denote its embedding in $\overline{U}_{2,w}$ by $\overline{u(0,y,0)}$.

If $y_{i,j}\in\calO_v$ for all $i,j$, then $\overline{u(0,y,0)}\in K_v$. Hence
\begin{equation}\label{inv01}
  f_{s,v}(\overline{u(0,y,0)}) = f_{s,v}(\id).
\end{equation}

For $r=(r_1,\dots,r_{k-2})\in\calO_v^{k-2}$, define
\[
  \nPlusKone(r) \;:=\; \prod_{i=1}^{k-2} X_{e_i+e_{k+1}}(r_i)\ \in \Sp_{2n}(F_v),
\]
where the product is taken in any fixed order. Since $\nPlusKone(r)\in U_k(F_v)$, it acts trivially on the left of $f_{s,v}$. Direct multiplication gives
\begin{equation}\label{r01}
  \nPlusKone(r) \begin{pmatrix}
    I_{k-2}&&&& \\
    &I_2&&& \\
    0&x&I_{2(n-k)}&& \\
    y&z&x^\ast&I_2& \\
    0&y^\ast&0&&I_{k-2}
  \end{pmatrix} (\nPlusKone(r))^{-1} = \begin{pmatrix}
    I_{k-2}&b&&& \\
    &I_2&&& \\
    0&x'&I_{2(n-k)}&& \\
    y&z&(x')^\ast&I_2&b^\ast \\
    0&y^\ast&0&&I_{k-2}
  \end{pmatrix} ,
\end{equation}
where $x\in\Mat_{2(n-k)\times2}$ and
\[
  x' = \begin{pmatrix}
    x_{1,1}+\sum_{i=1}^{k-2}y_{2,i}r_i&x_{1,2} +\sum_{i=1}^{k-2} y_{1,i}r_i \\
    \vdots&\vdots \\
    x_{2(n-k),1}&x_{2(n-k),2}
  \end{pmatrix} .
\]
Define the partial integral
\[
  G(y_1,y_2) = \int f_{s,v}(\overline{u(x,y,z)}) \psi_v(x_{1,1}-x_{2(n-k),2})\,dx\,dz ,
\]
where the integration is over the $x$- and $z$-coordinates in $\overline U_{2,w}(F_v)$. Then $\calF_v(f_{s,v},\id)=\int G(y_1,y_2)\,dy_1\,dy_2$.

The entries $b$ and $b^\ast$ in \eqref{r01} form an upper unipotent element in the $\GL_k$ factor of the Levi subgroup of $P_k(F_v)$. This element acts trivially on the left of the section by $P_k$-equivariance. Move $\nPlusKone(r)$ to the right and use right $K_v$-invariance. The changes of variables
\[
  x_{1,1}\mapsto x_{1,1}-\sum_{i=1}^{k-2}y_{2,i}r_i,\qquad x_{1,2}\mapsto x_{1,2}-\sum_{i=1}^{k-2}y_{1,i}r_i
\]
then give, for every $r\in\calO_v^{k-2}$,
\[
  G(y_1,y_2) = \psi_v\!\left(-\sum_{i=1}^{k-2}y_{2,i}r_i\right)G(y_1,y_2).
\]
Since $\psi_v$ has conductor $\calO_v$, this identity forces $G(y_1,y_2)=0$ unless $y_2\in\calO_v^{k-2}$. Indeed, if $y_{2,i}\notin\calO_v$, choose $r_i\in\calO_v$ with $\psi_v(-y_{2,i}r_i)\neq1$ and set the other entries of $r$ to zero. When $y_2\in\calO_v^{k-2}$, the corresponding unipotent element belongs to $K_v$. By \eqref{inv01} and right $K_v$-invariance, the integral over $y_2$ contributes volume $1$ under our normalization of Haar measure.

Applying the same argument to
\[
  \begin{pmatrix}
    I_{k-2}&0&^tr&&&& \\
    0&I_2&0&&&& \\
    0&0&1&&&& \\
    &&&I_{2(n-k-1)}&&& \\
    &&&&1&0&r^\ast \\
    &&&&0&I_2&0 \\
    &&&&0&0&I_{k-2}
  \end{pmatrix}
\]
removes the integration over $y_1$. Thus (\ref{int501}) becomes
\[
  \calF_v(f_{s,v},\id) = \int\limits_{\overline{U}^{(0)}_{2,w}(F_v)} f_{s,v}(u)\psi^{w^{-1}}_{\calO,v}(u)\,du,
\]
where $\overline{U}^{(0)}_{2,w}$ consists of the elements $\overline{u(x,0,z)}\in\overline{U}_{2,w}$. Define
\[
  \overline U'_{2,w} = \left\{ \overline{u(x,0,z)}\;\middle|\; x= \begin{pmatrix}
    0&x_{1,2} \\
    \vdots&\vdots \\
    0&x_{2(n-k),2}
  \end{pmatrix} , \quad z= \begin{pmatrix}
    0&z_3 \\
    0&0
  \end{pmatrix} \right\},
\]
and
\[
  \overline U''_{2,w} = \left\{ \overline{u(x,0,z)}\;\middle|\; x= \begin{pmatrix}
    x_{1,1}&0 \\
    \vdots&\vdots \\
    x_{2(n-k),1}&0
  \end{pmatrix} , \quad z= \begin{pmatrix}
    z_1&0 \\
    z_2&z_1
  \end{pmatrix} \right\}.
\]
Then $\overline{U}^{(0)}_{2,w}\simeq \overline{U}'_{2,w}\ltimes \overline{U}''_{2,w}$, and hence
\begin{equation}\label{int503}
  \calF_v(f_{s,v},\id) = \int\limits_{\overline{U}''_{2,w}(F_v)}\int\limits_{\overline{U}'_{2,w}(F_v)} f_{s,v}(u'u'')\psi^{w^{-1}}_{\calO,v}(u'u'')\,du'\,du''.
\end{equation}

For $u\in\overline{U}'_{2,w}(F_v)$, write $u=\overline{u(X',0,Z')}$, where
\[
  X'= \begin{pmatrix}
    0&r_1 \\
    0&r_2 \\
    \vdots&\vdots \\
    0&r_{2(n-k)}
  \end{pmatrix} , Z'= \begin{pmatrix}
    0&m+\displaystyle\sum_{i=1}^{n-k}r_ir_{2(n-k)+1-i} \\
    0&0
  \end{pmatrix} .
\]
Write this element as $u_{(r_1,r_2,\ldots,r_{2(n-k)};m)}$.

Set
\[
  \mathcal R = \{ -e_k+e_j: k+1\le j\le n\} \cup \{-2e_k\} \cup \{ -e_k-e_j: k+2\le j\le n\}.
\]
Let $N_{\mathcal R}(F_v):=\prod_{\alpha\in\mathcal R}X_\alpha(F_v)$, where the product is taken in the ordering compatible with the coordinates $(r_1,\ldots,r_{2(n-k)-1})$ appearing in $u_{(r_1,r_2,\ldots,r_{2(n-k)-1},0;m)}$. Hence we may write
\[
  u_{(r_1,r_2,\ldots,r_{2(n-k)-1},0;m)} = n_{\mathcal R}u', \qquad n_{\mathcal R}\in N_{\mathcal R}(F_v), \quad u'\in X_{-e_k-e_{k+1}}(F_v).
\]

Let $s_i$ denote the simple reflection corresponding to $\alpha_i$. Define
\[
  w_0 = s_{k+1}s_{k+2}\cdots s_{n-1}s_n s_{n-1}\cdots s_{k+1}s_k.
\]
The element $w_0$ takes $k$ to $-(k+1)$ and $k+1$ to $k$. Thus $w_0(\mathcal R)\subset\Sigma^+$, whereas $w_0(-e_k-e_{k+1})=e_{k+1}-e_k\in\Sigma^-$. The subgroup $X_{-e_k-e_{k+1}}$ therefore remains outside the intertwining integral. In a common region of absolute convergence, Fubini's theorem separates the $N_{\mathcal R}(F_v)$-coordinates. The character satisfies
\[
  \psi^{w^{-1}}_{\calO,v}(n_{\mathcal R}u'u'') = \psi^{w^{-1}}_{\calO,v}(u'u'')
\]
for $n_{\mathcal R}\in N_{\mathcal R}(F_v)$, $u'\in X_{-e_k-e_{k+1}}(F_v)$, and $u''\in \overline U''_{2,w}(F_v)$. Hence $(\ref{int503})$ becomes
\begin{equation}\label{int504}
  \int\limits_{\overline{U}''_{2,w}(F_v)} \int\limits_{X_{-e_k-e_{k+1}}(F_v)} \int\limits_{N_{\mathcal R}(F_v)} f_{s,v}(n_{\mathcal R}u'u'')\, \psi^{w^{-1}}_{\calO,v}(u'u'')\, dn_{\mathcal R}\,du'\,du''.
\end{equation}

Since $\mathcal R=\{\alpha\in\Sigma^-:w_0\alpha\in\Sigma^+\}$, the inner integral is the standard local intertwining integral for $w_0$. It defines $M(w_0,s):I(\Lambda_{s,v})\longrightarrow I(\Lambda_{s,v}^{w_0})$.

For a positive root $\beta$ with $w_0\beta\in\Sigma^-$, the corresponding rank-one integral converges absolutely when $|\Lambda_{s,v}(\beta^\vee(\varpi))|<1$. Put $a=s+\frac{k-1}{2}$. The positive roots with $w_0\beta\in\Sigma^-$ are
\[
  \{e_k-e_j:k+1\le j\le n\} \cup \{2e_k\} \cup \{e_k+e_j:k+2\le j\le n\}.
\]
The definition of $\Lambda_{s,v}$ gives
\[
  \begin{aligned}
    \Lambda_{s,v}((e_k-e_j)^\vee(\varpi)) &=\chi_v(\varpi)q^{-(a+n-j+1)} &&(k+1\le j\le n), \\
    \Lambda_{s,v}((2e_k)^\vee(\varpi)) &=\chi_v(\varpi)q^{-a}, \\
    \Lambda_{s,v}((e_k+e_j)^\vee(\varpi)) &=\chi_v(\varpi)q^{-(a+j-n-1)} &&(k+2\le j\le n).
  \end{aligned}
\]
Since $\chi_v$ is unitary, convergence requires $\Re(a+n-j+1)>0$ for $k+1\le j\le n$, $\Re(a)>0$, and $\Re(a+j-n-1)>0$ for $k+2\le j\le n$. The strongest of these inequalities is $\Re(a+k-n+1)>0$.

Thus $M(w_0,s)$ converges absolutely in a right half-plane. The identity extends meromorphically in $s$ by meromorphic continuation of local intertwining operators.

Let $f^{w_0}_{s,v}$ be the normalized spherical vector in $I(\Lambda_{s,v}^{w_0})$. The Gindikin--Karpelevich formula gives $M(w_0,s)f_{s,v} = c(w_0,s)f^{w_0}_{s,v}$, where
\[
  c(w_0,s) = \prod_{\substack{\beta\in\Sigma^+ \\
  w_0\beta\in\Sigma^-}} \frac{ 1-q^{-1}\Lambda_{s,v}(\beta^\vee(\varpi)) }{ 1-\Lambda_{s,v}(\beta^\vee(\varpi)) }.
\]
Substituting the coroot values above gives
\[
  c(w_0,s) = \frac{ 1-\chi_v(\varpi)q^{-\left(s+\frac{k-1}{2}-k+n+1\right)} }{ 1-\chi_v(\varpi)q^{-\left(s+\frac{k-1}{2}+k-n+1\right)} }.
\]
Therefore $(\ref{int504})$ equals
\[
  \frac{ 1-\chi_v(\varpi)q^{-\left(s+\frac{k-1}{2}-k+n+1\right)} }{ 1-\chi_v(\varpi)q^{-\left(s+\frac{k-1}{2}+k-n+1\right)} } \int\limits_{\overline{U}^{+}_{2,w}(F_v)} f^{w_0}_{s,v}(w_0u)\, \psi^{w^{-1}}_{\calO,v}(u)\,du,
\]
where
\[
  \overline{U}^{+}_{2,w} = \left\{\left. \begin{pmatrix}
    I_{k-2}&&&& \\
    &I_2&&& \\
    0&x&I_{2(n-k)}&& \\
    0&z&x^\ast&I_2& \\
    0&0&0&&I_{k-2}
  \end{pmatrix} \right\vert x = \begin{pmatrix}
    x_{1,1}&0 \\
    \vdots&\vdots \\
    x_{2(n-k)-1,1}&0 \\
    x_{2(n-k),1}&x_{2(n-k),2}
  \end{pmatrix} ,\; z= \begin{pmatrix}
    z_1'&0 \\
    z_2&z_1
  \end{pmatrix} \right\},
\]
where $z_1'=z_1+x_{1,1}x_{2(n-k),2}$.

Write $u\in\overline{U}^{+}_{2,w}$ as $u=u_{(x_{2(n-k),2};x_{1,1},\ldots,x_{2(n-k),1};c_1,c_2)}=\overline{u(X',0,Z')}$, where
\[
  X'= \begin{pmatrix}
    x_{1,1}&0 \\
    \vdots&\vdots \\
    x_{2(n-k)-1,1}&0 \\
    x_{2(n-k),1}&x_{2(n-k),2}
  \end{pmatrix} ,\quad Z'= \begin{pmatrix}
    c_1+x_{1,1}x_{2(n-k),2}&0 \\
    c_2 + \displaystyle\sum_{i=1}^{n-k}x_{i,1} x_{2(n-k)+1-i,1}&c_1
  \end{pmatrix} .
\]
If $x_{i,1},c_2\in\calO_v$ for $2\le i\le2(n-k)$, then $u=u_{(0;0,x_{2,1},\ldots,x_{2(n-k),1};0,c_2)}$ lies in $K_v$. Right $K_v$-invariance gives $f_{s,v}^{w_0}(w_0u)=f_{s,v}^{w_0}(w_0)$.

Let $(r,m;t):=(r_2,\ldots,r_{2(n-k)-1};\,m;\,t)\in\calO_v^{2(n-k)}$. For $k+2\le j\le n$, set $a_j:=r_{j-k}$ and $b_j:=r_{2n-j-k+1}$. Define
\[
  \begin{aligned}
    \nHeis(r,m;t) :=\;& \Bigl(\prod_{j=k+2}^{n} X_{e_{k}-e_j}(a_j)\Bigr)\, X_{2e_{k}}(m) \\
    &\times \Bigl(\prod_{j=k+2}^{n} X_{e_{k}+e_j}(b_j)\Bigr)\, X_{e_{k-1}+e_k}(t) \ \in \Sp_{2n}(F_v).
  \end{aligned}
\]
With this ordering, the long-root coordinate in the central factor becomes
\[
  m_r \;=\; m+\sum_{j=k+2}^{n} a_j b_j \;=\; m+\sum_{i=2}^{n-k} r_i\,r_{2(n-k)+1-i},
\]
by the commutator relations for the pairs $\bigl(e_{k}-e_j,\;e_{k}+e_j\bigr)$. The section $f^{w_0}_{s,v}$ is left-invariant under $\nHeis(r,m;t)$. Moving this element past $w_0$ gives
\[
  \nHeis(r,m;t)^{w_0} := w_0^{-1}\nHeis(r,m;t)w_0 = \begin{pmatrix}
    I_{k-2}&&&&&&&& \\
    &1&&&&t&&& \\
    &&1&0&0&0&0&& \\
    &&&1&r&m_r&0&t& \\
    &&&&I_{2(n-k-1)}&r^\ast&0&& \\
    &&&&&1&0&& \\
    &&&&&&1&& \\
    &&&&&&&1& \\
    &&&&&&&&I_{k-2} \\
  \end{pmatrix} .
\]

Matrix multiplication gives the changes of variables below after removing left factors whose conjugates by $w_0$ are upper unipotent.

First take $r=t=0$. The resulting change is
\[
  x_{1,1}\mapsto x_{1,1}+mx_{2(n-k),1}, \qquad c_2\mapsto c_2-mx_{2(n-k),1}^{\,2}.
\]
For fixed $x_{2,1},\ldots,x_{2(n-k),1}$, integrate over $x_{1,1},x_{2(n-k),2},c_1,c_2$. The change of variables has Jacobian one. Character orthogonality in $m\in\calO_v$ shows that this partial integral vanishes unless $x_{2(n-k),1}\in\calO_v$. Right $K_v$-invariance, together with the translation $c_2\mapsto c_2+2x_{1,1}x_{2(n-k),1}$, then removes the integration over $x_{2(n-k),1}$.

Next set $x_{2(n-k),1}=0$ and take $m=t=0$. The resulting change is
\[
  x_{1,1}\mapsto x_{1,1}+\sum_{i=2}^{2(n-k)-1}r_i x_{i,1},
\]
with all other coordinates unchanged. Applying character orthogonality to the same partial integration restricts $x_{i,1}$ to $\calO_v$ for $2\le i\le2(n-k)-1$. Right $K_v$-invariance removes these integrations.

After setting $x_{i,1}=0$ for $i\ge2$, take $r=m=0$. The resulting change is $x_{1,1}\mapsto x_{1,1}+c_2t$, with all other coordinates unchanged. For fixed $c_2$, integrate over $x_{1,1},x_{2(n-k),2},c_1$. Character orthogonality in $t\in\calO_v$ shows that this partial integral vanishes unless $c_2\in\calO_v$. Right invariance under $I_{2n}+c_2e_{2n-k+2,k-1}\in K_v$ then removes its integration. Each deleted compact integration contributes volume one. The resulting integral is
\begin{equation}\label{int508}
  \calF_v(f_{s,v},\id) = \frac{1-\chi_v(\varpi)q^{-(s+\frac{k-1}{2}-k+n+1)}}{1-\chi_v(\varpi)q^{-(s+\frac{k-1}{2}+k-n+1)}} \int\limits_{\overline{U}^{-}_{2,w}(F_v)} f^{w_0}_{s,v}(w_0u) \psi^{w^{-1}}_{\calO,v}(u)\,du,
\end{equation}
where
\begin{equation}\label{def502}
  \overline{U}^{-}_{2,w} = \left\{\left. \begin{pmatrix}
    I_{k-2}&&&& \\
    &I_2&&& \\
    0&x&I_{2(n-k)}&& \\
    0&z&x^\ast&I_2& \\
    0&0&0&&I_{k-2}
  \end{pmatrix} \right\vert x = \begin{pmatrix}
    x_{1,1}&0 \\
    0&0 \\
    \vdots&\vdots \\
    0&0 \\
    0&x_{2(n-k),2}
  \end{pmatrix} ,\; z= \begin{pmatrix}
    z_1'&0 \\
    0&z_1
  \end{pmatrix} \right\},
\end{equation}
where $z_1'=z_1+x_{1,1}x_{2(n-k),2}$.

Use the embedding $t:\GL_3(F_v)\to\Sp_{2n}(F_v)$ from \eqref{gl3em}. Let $U_{\GL_3}$ be the embedded unipotent radical of the standard Borel subgroup $B_{\GL_3}$. For $u\in\overline{U}^{-}_{2,w}(F_v)$ as in \eqref{def502}, we have
\[
  u = w_0^{-1} (w^k)^{-1} \, t\!\left( \begin{pmatrix}
    1&x_{2(n-k),2}&z_1 + x_{1,1}x_{2(n-k),2} \\
    &1&x_{1,1} \\
    &&1
  \end{pmatrix} \right) \, w^kw_0,
\]
where
\[
  w^k=s_ks_{k-1}s_k =t\!\left( \begin{pmatrix}
    &&1 \\
    &1& \\
    1&&
  \end{pmatrix} \right).
\]
This identifies $\overline{U}^{-}_{2,w}(F_v)$ with $U_{\GL_3}(F_v)$. Under this identification, $\psi^{w^{-1}}_{\calO,v}(u)=\psi_v(x_{1,1}-x_{2(n-k),2})$ is the Gelfand--Graev character. The integral in \eqref{int508} is therefore the unnormalized Jacquet integral of the spherical section in
\[
  \Ind_{B_{\GL_3}(F_v)}^{\GL_3(F_v)} \left(\chi_v|\cdot|^{s+\frac{k-3}{2}}\otimes |\cdot|^{k-n}\otimes \chi_v^{-1}|\cdot|^{-s-\frac{k-1}{2}}\right).
\]

The Casselman--Shalika formula \cite{CS80} gives
\[
  \begin{split}
    &\frac{1-\chi_v(\varpi)q^{-(s+\frac{k-1}{2}-k+n+1)}}{1-\chi_v(\varpi)q^{-(s+\frac{k-1}{2}+k-n+1)}} \int\limits_{U_{\GL_3}(F_v)} f^{w_0}_{s,v}\left(t\!\left( \begin{pmatrix}
      &&1 \\
      &1& \\
      1&&
    \end{pmatrix} \right)uw^kw_0\right) \psi^{w^{-1}}_{\calO,v}(u)\,du \\
    &= \left(1-\chi_v^2(\varpi)q^{-(2s+k-1)}\right) \left(1-\chi_v(\varpi)q^{-(s+\frac{k-1}{2}-k+n)}\right) \left(1-\chi_v(\varpi)q^{-(s+\frac{k-1}{2}+k-n+1)}\right) \\
    &\quad\times \frac{1-\chi_v(\varpi)q^{-(s+\frac{k-1}{2}-k+n+1)}}{1-\chi_v(\varpi)q^{-(s+\frac{k-1}{2}+k-n+1)}} f_{s,v}(w^kw_0) \\
    &= \left(1-\chi_v^2(\varpi)q^{-(2s+k-1)}\right) \left(1-\chi_v(\varpi)q^{-(s-\frac{k+1}{2}+n)}\right) \left(1-\chi_v(\varpi)q^{-(s-\frac{k-1}{2}+n)}\right) f_{s,v}(\id).
  \end{split}
\]
The last equality uses $w^kw_0\in K_v$, which gives $f_{s,v}(w^kw_0)=f_{s,v}(\id)$. The resulting expression is holomorphic in $s$, which proves the theorem.
\end{proof}

\subsection{The ramified and Archimedean cases}

At a ramified or Archimedean place $v$, consider the local Fourier coefficient
\[
  \calF_v(f_{s,v},g) = \int\limits_{\overline{U}_{2,w}(F_v)}f_{s,v}(u\iota(g))\psi^{w^{-1}}_{\calO,v}(u)\,du.
\]
\begin{proposition}\label{prop52}
Let $f_{s,v}\in I_{n,v}(s)$ be a section that varies holomorphically in $s\in\bbC$. Then $\calF_v(f_{s,v},\cdot)$ defines a meromorphic section of
\[
  I_{n-3,v}(s) = \Ind_{P_{k-2,n-3}(F_v)}^{\Sp_{2(n-3)}(F_v)} \chi_v|\det|^{s}\otimes \id_{\Sp_{2(n-k-1)}}.
\]
\end{proposition}
We use the following form of the Dixmier--Malliavin theorem \cite{DM78}.

\begin{lemma}
Let $G$ be a reductive group over an Archimedean field or a Lie group and $(\pi,V)$ a smooth representation of $G$. Then the space $V^\infty$ of smooth vectors is generated by the Schwartz algebra $\mathcal S(G)$ under convolution: for any $v\in V^\infty$ there exist $f_i\in\mathcal S(G)$ and $v_i\in V^\infty$ such that $v=\sum_i \pi(f_i)v_i$.
\end{lemma}
\begin{proof}[Proof of Proposition~\ref{prop52}]
The $P_{k-2,n-3}(F_v)$-equivariance of $\calF_v(f_{s,v},\cdot)$ has already been proved. It remains to establish convergence for $\Re(s)\gg0$ and meromorphic continuation. Since $\calF_v(f_{s,v},g) = \calF_v\bigl(\rho(\iota(g))f_{s,v},1\bigr)$ and right translation preserves smooth holomorphic sections, it is enough to consider the value at $g=1$.

Assume first that $v$ is non-Archimedean. We write $\overline U_{2,w}(F_v)$ as in \eqref{def501}. Write the $y$-block as $y=\binom{y_1}{y_2}$, with $y_1,y_2\in F_v^{k-2}$. Denote by $\overline U^{(0)}_{2,w}$ the subgroup obtained by setting these $y$-coordinates equal to zero.

To eliminate the $y_2$-coordinates, let $n(r)=\nPlusKone(r)$, where $r=(r_1,\ldots,r_{k-2})\in F_v^{k-2}$. This is the product of root subgroup elements used in the unramified computation. Since $f_{s,v}$ is smooth, there exists a compact open subgroup $K_f\subset \Sp_{2n}(F_v)$ such that $f_{s,v}(g\kappa)=f_{s,v}(g)$ for $\kappa\in K_f$. Choose $N\geq 0$ sufficiently large so that $n(r)\in K_f$ for all $r\in\varpi^N\calO_v^{\,k-2}$ and the $y$-root elements $\overline{u(0,y,0)}$ lie in $K_f$ whenever all coordinates of $y$ lie in $\varpi^N\calO_v$.

The calculation from the unramified case gives, for $r\in \varpi^N\calO_v^{\,k-2}$,
\[
  n(r)\,\overline{u(x,y,z)}\,n(r)^{-1}=b(r,x,y,z)\overline{u(x',y,z)},
\]
where $b(r,x,y,z)$ is an upper unipotent element of the inducing parabolic. It acts trivially on the section. The relevant change in the $x$-coordinates is $x'_{1,1}=x_{1,1}+\sum_{i=1}^{k-2}y_{2,i}r_i$.

The remaining changes in the coordinates are irrelevant for the character or are absorbed by a measure-preserving change of variables. Since $n(r)\in K_f$, right $K_f$-invariance and left $U_k(F_v)$-invariance imply that the inner integral, viewed as a function of $y_2$, satisfies
\[
  G(y_2) = \psi_v\!\left(-\sum_{i=1}^{k-2}y_{2,i}r_i\right)G(y_2) \qquad \left(r\in\varpi^N\calO_v^{\,k-2}\right).
\]
Therefore $G(y_2)=0$ unless $\psi_v\!\left(\sum_{i=1}^{k-2}y_{2,i}r_i\right)=1$ for all $r\in\varpi^N\calO_v^{\,k-2}$. Equivalently, $y_2\in \left(\varpi^N\calO_v^{\,k-2}\right)^\vee$. If the conductor of $\psi_v$ is $\mathfrak d_{\psi_v}$, this dual lattice is of the form $\left(\varpi^N\calO_v\right)^\vee = \varpi^{-N}\mathfrak d_{\psi_v}^{-1}\calO_v = \varpi^{-M}\calO_v$ for some integer $M$. Thus the $y_2$-integration is supported in the compact lattice $\varpi^{-M}\calO_v^{\,k-2}$.

Now decompose this lattice into finitely many cosets modulo $\varpi^N\calO_v^{\,k-2}$:
\[
  \varpi^{-M}{\calO_v}^{\,k-2} = \bigsqcup_{\eta\in \Xi_2} \left(\eta+\varpi^N\calO_v^{\,k-2}\right).
\]
Here $\Xi_2$ is a set of representatives for $\varpi^{-M}\mathcal O_v^{k-2}/\varpi^N\mathcal O_v^{k-2}$. For $y_2=\eta+y_2''$, with $y_2''\in\varpi^N\calO_v^{\,k-2}$, we have
\[
  \overline{u\!\left(x,\binom{y_1}{\eta+y_2''},z\right)} = \overline{u\!\left(x',\binom{y_1}{\eta},z'\right)}\kappa(y_2''),
\]
where $\kappa(y_2'')\in K_f$, the change of variables $(x,z)\mapsto(x',z')$ has Jacobian $1$, and the character $\psi^{w^{-1}}_{\calO,v}$ is unchanged. Hence, by right $K_f$-invariance, integration over each coset contributes a constant multiple of the integral with $y_2=\eta$. Consequently the $y_2$-integration is absorbed into a finite linear combination of right translates of $f_{s,v}$.

Repeating the same argument for the $y_1$-coordinates, using the corresponding root subgroups as in the unramified case, we obtain finitely many elements $\zeta_j\in \overline U_{2,w}(F_v)$ and constants $c_j$ such that, for $\Re(s)\gg0$,
\[
  \calF_v(f_{s,v},1) = \sum_j c_j \int_{\overline U^{(0)}_{2,w}(F_v)} (\rho(\zeta_j)f_{s,v})(u)\, \psi^{w^{-1}}_{\calO,v}(u)\,du .
\]
Under the embedding of the subgroup $\Sp_{2(n-k+2)}$ on the relevant coordinates, $\overline U^{(0)}_{2,w}$ is an ordered product of root subgroups contained in the opposite unipotent radical of the parabolic $P_{2,n-k+2}$. With the character $\psi^{w^{-1}}_{\calO,v}$, the integral over $\overline U^{(0)}_{2,w}$ is one of the standard generalized Jacquet integrals attached to this degenerate principal series. Each $\rho(\zeta_j)f_{s,v}$ is again a smooth holomorphic section of the same induced representation. Therefore the integrals converge absolutely for $\Re(s)\gg0$ and admit meromorphic continuation in $s$ by the standard analytic theory of local Jacquet integrals \cite{J67}. Since the formula is a finite sum, the same conclusion holds for $\calF_v(f_{s,v},1)$. This proves the non-Archimedean case.

Assume now that $v$ is Archimedean. Put $D:=F_v^{k-2}$. For $r\in D$, write $n(r)=\nPlusKone(r)\in U_k(F_v)$. By the Dixmier--Malliavin theorem, applied to the smooth right action of the Lie group $D$, it is enough to treat sections of the form $f_{s,v}=\rho(\phi)\varphi_s$, where $\phi\in \mathcal S(D)$, $\varphi_s$ is a smooth holomorphic section, and
\[
  (\rho(\phi)\varphi_s)(g) = \int_D \phi(r)\,\varphi_s(g\,n(r))\,dr .
\]

Put
\[
  \begin{gathered}
    Y:=\Mat_{2\times(k-2)}(F_v),\qquad Y_2:=F_v^{k-2}, \\
    Y^{(2)}:=\left\{\binom{0}{y_2}:y_2\in Y_2\right\}\subset Y.
  \end{gathered}
\]
Choose a closed subgroup $\overline U'_{2,w}(F_v)\subset \overline U_{2,w}(F_v)$ such that multiplication induces a diffeomorphism
\[
  \overline U'_{2,w}(F_v)\times u'_2(Y_2) \longrightarrow \overline U_{2,w}(F_v),\qquad (u',u'_2(y_2))\longmapsto u'u'_2(y_2),
\]
where $u'_2(y_2):=\overline{u\!\left(0,\binom{0}{y_2},0\right)}$.

We choose Haar measures so that $du=du'\,dy_2$. Since $\psi^{w^{-1}}_{\calO,v}$ is trivial on $u'_2(Y_2)$, we have $\psi^{w^{-1}}_{\calO,v}\bigl(u'u'_2(y_2)\bigr) = \psi^{w^{-1}}_{\calO,v}(u')$. For $\Re(s)\gg0$, the integral is absolutely convergent. The rapid decay of $\phi$ then allows us to apply Fubini's theorem:
\begin{align*}
  \calF_v(\rho(\phi)\varphi_s,1) &= \int_{\overline U_{2,w}(F_v)} (\rho(\phi)\varphi_s)(u)\, \psi^{w^{-1}}_{\calO,v}(u)\,du \\
  &= \int_{\overline U'_{2,w}(F_v)} \int_{Y_2} \int_D \phi(r)\, \varphi_s\!\bigl(u'u'_2(y_2)n(r)\bigr)\, dr\, \psi^{w^{-1}}_{\calO,v}(u')\,dy_2\,du'.
\end{align*}
The left $U_k(F_v)$-invariance of $\varphi_s$ gives
\[
  \varphi_s\!\bigl(u'u'_2(y_2)n(r)\bigr) = \varphi_s\!\bigl(n(r)^{-1}u'u'_2(y_2)n(r)\bigr).
\]
The same calculation as in the non-Archimedean case shows that conjugation by $n(r)$ preserves the $y_2$-coordinate and changes the relevant $x$-coordinate by $\langle y_2,r\rangle=\sum_{i=1}^{k-2}y_{2,i}r_i$. Thus
\[
  \psi^{w^{-1}}_{\calO,v}\bigl(n(r)^{-1}u'u'_2(y_2)n(r)\bigr) =\psi^{w^{-1}}_{\calO,v}(u')\,\psi_v(-\langle y_2,r\rangle).
\]
A change of variables gives
\begin{align*}
  \calF_v(\rho(\phi)\varphi_s,1) &= \int_{\overline U'_{2,w}(F_v)} \int_{Y_2} \int_D \phi(r)\, \psi_v(\langle y_2,r\rangle)\, \varphi_s\!\bigl(u'u'_2(y_2)\bigr)\, dr\, \psi^{w^{-1}}_{\calO,v}(u')\,dy_2\,du' .
\end{align*}

Define
\[
  \widehat\phi(y_2) := \int_D \phi(r)\,\psi_v(\langle y_2,r\rangle)\,dr .
\]
Then $\widehat\phi\in\mathcal S(Y_2)$. Hence
\[
  \calF_v(\rho(\phi)\varphi_s,1) = \int_{\overline U'_{2,w}(F_v)} \psi^{w^{-1}}_{\calO,v}(u') \left( \int_{Y_2} \widehat\phi(y_2)\, \varphi_s\!\bigl(u'u'_2(y_2)\bigr)\,dy_2 \right) du' .
\]
Set
\[
  \varphi_{1,s}(u') := \int_{Y_2} \widehat\phi(y_2)\, \varphi_s\!\bigl(u'u'_2(y_2)\bigr)\,dy_2 .
\]
The rapid decay of $\widehat\phi$ implies that $\varphi_{1,s}$ is again obtained from $\varphi_s$ by Schwartz convolution. It is therefore smooth, depends holomorphically on $s$, and has uniform moderate growth on vertical strips. Thus the $y_2$-integration has been absorbed into the section.

Repeating the same argument for the remaining $y_1$-coordinates, using the corresponding root subgroups as in the unramified case, reduces the original local Fourier coefficient to an integral of the form
\[
  \int_{\overline U^{(0)}_{2,w}(F_v)} \varphi'_{s}(u)\, \psi^{w^{-1}}_{\calO,v}(u)\,du ,
\]
where $\varphi'_s$ is obtained from the original section by finitely many Schwartz convolutions and therefore remains a smooth holomorphic family of moderate growth. This is again a generalized Jacquet integral for the degenerate principal series. By \cite{J67}, it converges absolutely for $\Re(s)\gg0$ and admits meromorphic continuation in $s$. Hence $\calF_v(f_{s,v},\cdot)$ defines a meromorphic section of $I_{n-3,v}(s)$.
\end{proof}

\begin{proposition}\label{prop4}
Fix a place $v$ and $s_0\in\mathbb C$ with $\Re(s_0)\geq0$. Assume that, in fixed compact pictures, the continued local Fourier coefficients form a meromorphic family of continuous operators
\[
  \calF_{v,s}:I_{n,v}(s)\longrightarrow I_{n-3,v}(s)
\]
which is holomorphic at $s=s_0$. Then $\calF_{v,s_0}$ is surjective if $v$ is non-Archimedean, and has dense image in the smooth Fr\'echet topology if $v$ is Archimedean.
\end{proposition}

\begin{proof}
Put $a=k-2$ and $q=n-k-1$, and write
\[
  \begin{aligned}
    G&=\Sp_{2n}(F_v),&\qquad P&=P_k(F_v), \\
    G'&=\Sp_{2(n-3)}(F_v),& P'&=P_{k-2,n-3}(F_v).
  \end{aligned}
\]
Let $U_{2,w}^{\mathrm{c}}\subset U_2$ be the explicit subgroup complement representing the quotient $U_2^w\backslash U_2$, so that $wU_{2,w}^{\mathrm{c}}w^{-1}=\overline U_{2,w}$ with $\overline U_{2,w}$ as in \eqref{def501}. Let
\[
  \Omega_{2q}= \begin{pmatrix}
    0&J_q \\
    -J_q&0
  \end{pmatrix} .
\]
For $Y\in\Mat_{2q\times a}(F_v)$ and $Z\in\Mat_a(F_v)$ with $J_aZ=Z^tJ_a$, put
\[
  \begin{gathered}
    Y^*=-J_aY^t\Omega_{2q}, \\
    C(Y,Z)=Z-\frac12J_aY^t\Omega_{2q}Y=Z+\frac12Y^*Y.
  \end{gathered}
\]
Then
\[
  \bar u(Y,Z)= \begin{pmatrix}
    I_a&0&0 \\
    Y&I_{2q}&0 \\
    C(Y,Z)&Y^*&I_a
  \end{pmatrix} \in\bar U'.
\]

Let $w_{P'}$ be the relative open-cell Weyl representative for $P'\backslash G'$. Since $w_{P'}^{-1}\bar U'w_{P'}=U'$, put $u^+(Y,Z)=w_{P'}^{-1}\bar u(Y,Z)w_{P'}\in U'$ and define
\[
  \kappa_+(Y,Z)=\bar u(Y,Z)w_{P'}=w_{P'}u^+(Y,Z).
\]
Then $(Y,Z)\longmapsto P'\kappa_+(Y,Z)$ parametrizes the full open Bruhat cell of $P'\backslash G'$. For $u\in U_{2,w}^{\mathrm{c}}(F_v)$, define
\[
  \Phi_+(u,Y,Z)=P\,wu\,\iota_c\!\left(\bar u(Y,Z)w_{P'}\right).
\]
Put $W=w\iota_c(w_{P'})$. A direct calculation of the relative inversion roots shows that the root coordinates contributed by $U_{2,w}^{\mathrm{c}}$ and $\iota_c(U')$ form a root-closed subgroup of the $W$-Bruhat cell. The remaining roots transverse to this subgroup are $\{e_3+e_c:4\leq c\leq k+1\}$.

Consequently $\Phi_+$ is an algebraic isomorphism onto the subvariety $\mathcal Z_W \subset P\backslash G$ obtained by setting these transverse root coordinates equal to zero. Hence $\mathcal Z_W$ is closed in the $W$-Bruhat cell and therefore locally closed in $P\backslash G$.

It follows from the local product structure that every compactly supported smooth section on $U_{2,w}^{\mathrm{c}}(F_v)\times\bar U'(F_v)$, viewed as a section on $\mathcal Z_W$ through $\Phi_+$, extends to a compactly supported smooth section on $P\backslash G$. At a non-Archimedean place this follows from compact-open product charts; at an Archimedean place it follows from a tubular neighborhood and a compactly supported cutoff.

Write
\[
  I'_v(s)=I_{n-3,v}(s) = \Ind_{P'}^{G'} \left( \chi_v|\det|^s \otimes1_{\Sp_{2(n-k-1)}(F_v)} \right),
\]
and realize its smooth contragredient as
\[
  J'_v(s) = \Ind_{P'}^{G'} \left( \chi_v^{-1}|\det|^{-s} \otimes1_{\Sp_{2(n-k-1)}(F_v)} \right).
\]
All inductions are normalized.

Let $K'\subset G'$ be a maximal compact subgroup. In the fixed compact pictures, the standard pairing $[F,\xi]_s = \int_{K'}\langle F(k),\xi(k)\rangle\,dk$ is nondegenerate, $G'$-invariant, and independent of $s$. It is equivalent to work with the original local model
\[
  \mathcal A_{v,s}(f)(g) = \int_{U_{2,w}^{\mathrm{c}}(F_v)} f\bigl(wu\iota_c(g)\bigr) \psi_{\calO,v}(u)\,du,
\]
since it differs from the conjugated model by an invertible right translation of the source.

Let $\xi_0\in J'_v(s_0)$ annihilate the image of $\mathcal A_{v,s_0}$. We prove that $\xi_0=0$. On the open cell, the compact-picture pairing takes the form
\[
  [F,\xi]_s = \int_{\bar U'(F_v)} \left\langle F(\kappa_+(Y,Z)), \xi(\kappa_+(Y,Z)) \right\rangle \,d\mu(Y,Z),
\]
where $d\mu$ is a fixed nowhere-vanishing smooth density.

Let $\varphi\in C_c^\infty(\bar U'(F_v))$ be arbitrary. Choose $h\in C_c^\infty(U_{2,w}^{\mathrm{c}}(F_v))$ such that $\int_{U_{2,w}^{\mathrm{c}}(F_v)}h(u)\psi_{\calO,v}(u)\,du=1$. There is a compactly supported source section $f_0\in I_{n,v}(s_0)$ satisfying $f_0(\Phi_+(u,Y,Z))=h(u)\varphi(Y,Z)$, with joint compact support in $(u,Y,Z)$.

Extend $f_0$ and $\xi_0$ to flat entire families $f_s$ and $\xi_s$ in the corresponding compact pictures. Define $B_{\mathrm{cont}}(s) = [\mathcal A_{v,s}(f_s),\xi_s]_s$ and
\[
  B_{\mathrm{raw}}(s) = \int_{\bar U'(F_v)} \int_{U_{2,w}^{\mathrm{c}}(F_v)} \psi_{\calO,v}(u) \left\langle f_s(\Phi_+(u,Y,Z)), \xi_s(\kappa_+(Y,Z)) \right\rangle \,du\,d\mu(Y,Z).
\]
The joint compact support makes $B_{\mathrm{raw}}(s)$ entire. By the hypotheses on $\mathcal A_{v,s}$, the function $B_{\mathrm{cont}}(s)$ is meromorphic and regular at $s=s_0$. For $\Re(s)\gg0$, the defining integral gives $B_{\mathrm{cont}}(s)=B_{\mathrm{raw}}(s)$. The meromorphic identity theorem extends this equality to $s=s_0$. Since $\xi_0$ annihilates the image, $B_{\mathrm{cont}}(s_0)=0$. Using the prescribed restriction of $f_0$ and the normalization of $h$, we obtain
\[
  \int_{\bar U'(F_v)} \left\langle \varphi(Y,Z), \xi_0(\kappa_+(Y,Z)) \right\rangle \,d\mu(Y,Z)=0.
\]
Since $\varphi$ is arbitrary, $\xi_0$ vanishes on the full open Bruhat cell. That cell is dense and $\xi_0$ is smooth, hence $\xi_0=0$. At a non-Archimedean place, the image is a $G'$-subrepresentation of the admissible finite-length representation $I'_v(s_0)$. A proper image would give a nonzero smooth contragredient annihilator through an irreducible quotient, contradicting the preceding paragraph. Hence the map is surjective.

At an Archimedean place, if the image were not dense, Hahn--Banach would give a nonzero continuous annihilator. Convolution with $C_c^\infty(G')$ and the use of a smooth approximate identity would produce a nonzero smooth contragredient annihilator, again a contradiction. Therefore the image is dense.
\end{proof}

\section{Proof of the main descent theorem and applications}\label{last}

\subsection{Proof of the main descent theorem}

\begin{proof}[Proof of Theorem~\ref{thm401}]
In the initial absolute-convergence half-plane, the unfolding proved above identifies $F_{\psi_{\calO}}(E_{f_s})$ with the Eisenstein series attached to the descended section $g\longmapsto\calF(f_s,g)$.

Suppose that $f_s=\otimes_v'f_{s,v}$ is a factorizable standard section. Choose a finite set $S$ containing the Archimedean places and all places at which the data are ramified. For $v\notin S$, let $f_{s,v}^{\circ}$ and $f_{s,v}^{\prime\circ}$ denote the normalized source and target spherical sections. By Theorem~\ref{thm501},
\[
  \calF_v(f_{s,v}^{\circ},g_v) =c_v(s)f_{s,v}^{\prime\circ}(g_v),
\]
where
\[
  \begin{aligned}
    c_v(s)={}&L_v(\chi_v^2,2s+k-1)^{-1} \\
    &{}\times L_v\left(\chi_v,s+n-\frac{k+1}{2}\right)^{-1} L_v\left(\chi_v,s+n-\frac{k-1}{2}\right)^{-1}.
  \end{aligned}
\]
Consequently
\[
  \calF(f_s)=C^S(s) \left(\bigotimes_{v\notin S}'f_{s,v}^{\prime\circ}\right) \otimes\left(\bigotimes_{v\in S}\calF_v(f_{s,v})\right),
\]
where
\[
  C^S(s)= \frac{1}{ L^S(\chi^2,2s+k-1) L^S\left(\chi,s+n-\frac{k+1}{2}\right) L^S\left(\chi,s+n-\frac{k-1}{2}\right)}.
\]
Proposition~\ref{prop52} gives the local meromorphic continuations at the places in $S$. The unramified formula gives the remaining factors. The identity therefore extends meromorphically. Thus $F_{\psi_{\calO}}(E_{f_s})$ is the Eisenstein series attached to $\calF(f_s,\cdot)\in I_{n-3}(s)$.

At a parameter $s_0$ with $\Re(s_0)\geq0$ where the relevant local operator families are regular, Proposition~\ref{prop4} gives exact local preimages at finite places and dense local image at Archimedean places. The general right-$K$-finite case follows locally from finite linear combinations of factorizable flat sections with holomorphic coefficients.
\end{proof}

\subsection{Maximal Fourier coefficients}

For $a\geq2$, put
\begin{equation}\label{poles}
  \Lambda_a= \begin{cases}
    \{j-\tfrac12:1\leq j\leq a/2\},&a\text{ even}, \\[2pt]
    \{j:1\leq j\leq(a-1)/2\},&a\text{ odd}.
  \end{cases}
\end{equation}
\begin{remark}\label{partition}
Assume that $\chi$ is a nontrivial quadratic Hecke character. Put $r=n-k$ and $a=3k-2n$. If $r\geq1$ and $a\geq2$, then the descent may be iterated $r$ times:
\[
  (n,k)\longmapsto(n-3,k-2)\longmapsto\cdots\longmapsto(a,a).
\]
Thus the last member of the tower is the Siegel Eisenstein series on $\Sp_{2a}$ with the same parameter $s$. For $s\in\Lambda_a$, Ginzburg--Soudry \cite[Theorem~2.2(III)]{GS221} show that the corresponding terminal residual representation has maximal partition $[2^a]$.

The iterated descent and the identity in Jiang--Liu \cite[Proposition~3.3]{JL15} suggest adjoining a pair $3^2$ at each step. This gives the candidate partition
\[
  [3^{2r}2^a]=[3^{2(n-k)}2^{3k-2n}].
\]
This is the Richardson partition of $P_k$, by the induction rule for type $C$ \cite[Corollary~7.3.4(ii)]{CM93}. The local wave-front bounds of \cite{CK25}, together with \cite{Lapid08} and \cite{GGS17}, therefore provide a route to showing that this partition is maximal for suitable Laurent coefficients of the original Eisenstein series.
\end{remark}

\bigskip
\begingroup
\footnotesize
\noindent\textsc{Alan Xuelun Hou}\\
Department of Mathematics, Brandeis University, Waltham, Massachusetts, USA\\
\textit{Email address:} \href{mailto:hou@brandeis.edu}{\texttt{hou@brandeis.edu}}

\medskip
\noindent\textsc{Yusheng Lei}\\
Shanghai, China\\
\textit{Email address:} \href{mailto:leiusher@gmail.com}{\texttt{leiusher@gmail.com}}
\endgroup
\end{document}